\documentclass[10pt,a4paper]{amsart}
\usepackage[utf8]{inputenc}
\usepackage{amsmath}
\usepackage{amsfonts}
\usepackage{amssymb}
\usepackage{bbm}
\usepackage{mathrsfs}
\usepackage[all]{xy}
\usepackage[left=2cm,right=2cm,top=2cm,bottom=2cm]{geometry}
\author{Geoffrey Powell}
\title{The primitive filtration of the Leibniz complex}

\address{Univ Angers, CNRS, LAREMA, SFR MATHSTIC, F-49000 Angers, France}
\email{Geoffrey.Powell@math.cnrs.fr}
\urladdr{http://math.univ-angers.fr/~powell/}
\date{}

\newtheorem{THM}{Theorem}
\newtheorem{COR}[THM]{Corollary}

\newtheorem{thm}{Theorem}[section]
\newtheorem{prop}[thm]{Proposition}
\newtheorem{cor}[thm]{Corollary}
\newtheorem{lem}[thm]{Lemma}
\theoremstyle{definition}
\newtheorem{defn}[thm]{Definition}
\newtheorem{exam}[thm]{Example}
\theoremstyle{remark}
\newtheorem{rem}[thm]{Remark}
\newtheorem{nota}[thm]{Notation}
\newtheorem{hyp}[thm]{Hypothesis}
\newcommand{\ld}{{}_{(1)}\Delta}
\newcommand{\dlie}{d^{\mathrm{Lie}}}
\newcommand{\filt}{\mathsf{f}}
\newcommand{\kring}{\mathbbm{k}}
\newcommand{\g}{\mathfrak{g}}
\newcommand{\h}{\mathfrak{h}}
\newcommand{\glie}{{\g_{\mathrm{Lie}}}}
\newcommand{\alg}{\mathrm{Alg}}
\newcommand{\leibalg}{\alg^{\mathsf{Leib}}}
\newcommand{\leibalgf}{\alg^{\mathsf{Leib}, \mathrm{f}}}
\newcommand{\liealg}{\alg^{\mathsf{Lie}}}
\newcommand{\leibopd}{\mathsf{Leib}}
\newcommand{\lieopd}{\mathsf{Lie}}
\newcommand{\op}{^\mathrm{op}}
\newcommand{\nat}{\mathbb{N}}
\newcommand{\zed}{\mathbb{Z}}
\newcommand{\ce}{\mathrm{CE}}
\renewcommand{\hom}{\mathrm{Hom}}
\newcommand{\id}{\mathrm{Id}}
\newcommand{\sym}{\mathfrak{S}}
\newcommand{\ob}{\mathrm{Ob}}
\newcommand{\coker}{\mathrm{coker}}
\newcommand{\cop}{^{\ +\circlearrowright}}
\newcommand{\dbar}{\overline{d}}
\newcommand{\dleib}{d_{\mathrm{Leib}}}
\newcommand{\lp}{\star}
\newcommand{\opend}{\mathsf{End}}
\newcommand{\nt}{\mathsf{Nat}}
\newcommand{\cat}{\mathbf{cat}\ }
\newcommand{\vect}{\mathsf{Vect}_\kring}
\newcommand{\vectf}{{\vect^{\mathrm{f}}}}
\newcommand{\opd}{\mathscr{O}}
\newcommand{\opdalg}{\alg^\opd}
\newcommand{\opdalgf}{\alg^{\opd, \mathrm{f}}}
\newcommand{\sgn}{\mathrm{sgn}}

\begin{document}

\footnotetext{https://orcid.org/0000-0003-2564-1202}

\begin{abstract}
Pirashvili exhibited a small subcomplex of the Leibniz complex $(T(s \g), \dleib)$ of a Leibniz algebra $\g$. The main result of this paper generalizes this result to show that the primitive filtration of $T(s\g)$ provides an increasing, exhaustive filtration of the Leibniz complex by subcomplexes, thus establishing a conjecture due to Loday. The associated spectral sequence is used to give a new proof of Pirashvili's conjecture that, when $\g$ is a free Leibniz algebra, the homology of the Pirashvili complex is zero except in degree one.  

This result is then used to show that the desuspension of the Pirashvili complex carries a natural $L_\infty$-structure that induces the natural Lie algebra structure on the homology of the complex in degree zero. 
\end{abstract}

\maketitle
\section{Introduction}
\label{sect:intro}

Leibniz algebras are a `non-commutative geometry' avatar of Lie algebras: the binary operation is not required to be anticommutative and only the right Leibniz identity is imposed. In particular, any Lie algebra is a Leibniz algebra. In the other direction, given a Leibniz algebra $\g$, one can form the associated Lie algebra $\glie$. This is given as the quotient $\g \twoheadrightarrow \glie$  imposing anticommutativity; the  functor $\g \mapsto \glie$ is the left adjoint to the inclusion of Lie algebras in Leibniz algebras.

It was a fundamental observation of Loday's that the Chevalley-Eilenberg complex for Lie algebras has a counterpart for Leibniz algebras, $(T(s \g), \dleib)$, where $T(s\g)$ is the tensor algebra on the homological suspension of $\g$. This comes equipped with a morphism of complexes to the Chevalley-Eilenberg complex of $\glie$. This induces a comparison map from the Leibniz homology of $\g$ to the Lie algebra homology of $\glie$
\[
HL_*(\g; \kring)
\rightarrow 
H_*(\glie; \kring)
\]
with trivial coefficients. The homological relationships associated to the passage between Leibniz algebras and Lie algebras are of significant interest \cite{Cuv,MR1710505}; more generally, the relationship between Leibniz and Lie algebras has been  studied in the homotopical context in \cite{MR1638169}.

Henceforth   $\kring$ is taken to be a field of characteristic zero. In foundational work dating from 2002, Pirashvili \cite{2019arXiv190400121P} showed  that the Leibniz differential restricts to define a subcomplex $(\lieopd (s\g), \dleib) \subset (T( s\g), \dleib)$ of the Leibniz complex, where $\lieopd(-)$ is the free Lie algebra functor. The starting point of this paper is an alternative proof of this fact that is given in Proposition \ref{prop:lie_stable}.

 This feeds into the study of the primitive filtration $\filt_n T (s \g)$ associated to the primitively-generated Hopf algebra structure on $T(s\g)$. The following establishes a conjecture due to Loday \cite[Conjecture 2]{2019arXiv190400121P}:

\begin{THM}
[Theorem \ref{thm:prim_filt_stable}]
For $\g$ a Leibniz algebra and $n \in \nat$, $\filt_n T(s \g)$ is a subcomplex of the Leibniz complex $(T(s\g),d)$. This defines an increasing, exhaustive filtration 
\[
\filt_{0} T(s\g) = \kring \subset 
\ldots \subset \filt_n T(s\g) \subset \filt_{n+1} T(s\g) \subset \ldots \subset T(s\g).
\]
\end{THM}

The associated spectral sequence leads to a new proof of \cite[Conjecture 1]{2019arXiv190400121P}, namely that, for $\g = \leibopd (V)$ the free Leibniz algebra on a $\kring$-vector space $V$, the homology of $\lieopd( s\g)$ is concentrated in degree one: 

\begin{THM}[Theorem \ref{thm:liezation}]
\label{THM:liez}
For $V$ a $\kring$-vector space, 
\[
H_* (\lieopd (s \leibopd (V)), \dleib)\cong 
\left\{ 
\begin{array}{ll}
\lieopd (V) & *=1 \\
0 & \mbox{otherwise.}
\end{array}
\right.
\]
\end{THM}

This result was first proved by Mostovoy (see \cite[Proposition 13]{2019arXiv191003754M}) by different methods, based on his interpretation of Leibniz algebras in terms of differential graded Lie algebras.

The argument used in the proof of Theorem \ref{THM:liez} generalizes to relate the behaviour of the comparison map $HL_*(\g; \kring)
\rightarrow 
H_*(\glie; \kring)$ with that of the homology of the complex $(\lieopd (s\g), \dleib)$ (see Theorem \ref{thm:comparison}).

The significance  of Theorem \ref{THM:liez} is illustrated by 
the following Corollary, due  to Pirashvili \cite{2019arXiv190400121P}, who showed that it is a consequence of 
\cite[Conjecture 1]{2019arXiv190400121P}:

\begin{COR}
The homology $H_{*+1} (\lieopd (s\g), \dleib)$ is isomorphic to the Quillen derived functors of $(-)_{\mathrm{Lie}}$ applied to the Leibniz algebra $\g$. 
\end{COR}

Pirashvili asked the author whether the desuspension $s^{-1} \lieopd (s \g)$ carries a {\em natural} $L_\infty$-structure that extends the Leibniz differential and which induces the Lie structure on $\glie$. The remainder of the paper is devoted to the construction of such a structure. 

\begin{THM}[Theorem \ref{thm:main}]
\label{THM:Linf}
For $\g$ a Leibniz algebra, there is a natural $L_\infty$-structure on the desuspension of the Pirashvili complex $s^{-1}(\lieopd (s\g) , \dleib)$ that induces the canonical Lie algebra structure on $\glie$ in homology (up to sign).
\end{THM}

This result gives an algebraic structure on $s^{-1} \lieopd (s \g)$ that extends the antisymmetrization of the Leibniz structure on $\g$ given by $(x,y) \mapsto  x \lp y := xy -yx$.  Given this initial data, the natural $L_\infty$-structure is unique up to natural $L_\infty$-isomorphism, as stated in Theorem \ref{thm:main}.

These results were proved independently of Mostovoy's work \cite{2019arXiv191003754M}.

\subsection{Notation and basics}
The non-negative integers are denoted by $\nat$ (so that $0 \in \nat$).

Throughout $\kring$ is a field of characteristic zero. $\vect$ denotes the category of $\kring$-vector spaces and $\vectf$ the full subcategory of finite-dimensional spaces.

Graded $\kring$-vector spaces are homologically $\zed$-graded (and are usually concentrated in non-negative degrees). The category of graded $\kring$-vector spaces is taken to be symmetric monoidal for the usual tensor product  and with symmetry invoking Koszul signs. 

 For $W$ a $\zed$-graded $\kring$-vector space, $sW$ denotes the homological suspension of $W$ (i.e., $(sW)_{n} = W_{n-1}$, for $n \in \zed$). In particular, for $V$ a $\kring$-vector space (considered as concentrated in homological degree zero), $sV$ is the graded $\kring$-vector space consisting of $V$ placed in homological degree one.  For example, a morphism of degree $-1$ between graded vector spaces $U,V$ is equivalent to a morphism of graded vector spaces (i.e., preserving degree) $U \rightarrow sV$.  

\begin{nota}
For $n \in \nat$, the symmetric group on $n$ letters is denoted $\sym_n$ and the signature representation of $\sym_n$ by $\sgn_n$, which has underlying vector space $\kring$. This is also considered as a $\sym_n\op$-module (i.e., a right $\sym_n$-module).
\end{nota}

\begin{nota}
For $W$ a graded $\kring$-vector space and $n \in \nat$:
\begin{enumerate}
\item 
$T^n (W):= W^{\otimes n}$, the $n$th tensor power, equipped with the action of the symmetric group $\sym_n$ by place permutations (with Koszul signs);
\item 
$T(W) := \bigoplus _{n\in \nat} T^n (W)$; 
\item 
$\Gamma^n (W) :=  T^n (W) ^{\sym_n}$, the $n$th divided power;
\item
$S^n (W):=  T^n (W)_{\sym_n}$, the $n$th symmetric power.
\end{enumerate}
\end{nota}

\begin{rem}
\ 
\begin{enumerate}
\item 
The graded vector space $T(W)$ is the underlying space of the tensor algebra on $W$ (the free, unital associative algebra on the graded $\kring$-vector space $W$).
\item 
The formation of the  invariants $T^n (W) ^{\sym_n}$ and the  coinvariants $ T^n (W)_{\sym_n}$ takes into  account Koszul signs. For example, for $n \in \nat$ and $V$ a $\kring$-vector space (in degree zero), there is a natural isomorphism 
$$
 S^n (sV) 
 \cong
 s^n \Lambda^n (V), 
$$
 where $\Lambda^n (V)$ is the $n$th exterior power.
\end{enumerate}
\end{rem}

The following standard result for modules over a Lie algebra is used: 

\begin{lem}
\label{lem:lie_symm_mon}
For $\h$ a Lie algebra over $\kring$ (concentrated in degree zero), the category of $\zed$-graded right $\h$-modules is a symmetric monoidal category for the graded tensor product, with symmetry invoking Koszul signs; the unit is $\kring$ equipped with the trivial $\h$-module structure.

Explicitly, for right $\h$-modules $M, N$ and $m \in M$, $n\in N$, $h\in \h$, the action on $M \otimes N$ is given by
\[
(m \otimes n) h = mh \otimes n + m \otimes nh.
\]
\end{lem}

\begin{rem}
The hypothesis that $\h$ is concentrated in degree zero ensures that the  action on the tensor product does not involve Koszul signs.
\end{rem}

\tableofcontents
\part{The primitive filtration}

\section{Background}
\label{sect:background}
 
\subsection{Introducing Leibniz algebras}
This Section recalls the definition of a (right) Leibniz algebra and the construction of the associated Lie algebra.  

\begin{defn}
\label{defn:Leibniz}
A Leibniz algebra  is a $\kring$-vector space $\g$ equipped with a product $\g \otimes \g \rightarrow \g$, $x \otimes y \mapsto xy$, that satisfies the Leibniz relation:
\[
x(yz) = (xy)z - (xz) y.
\]
A morphism of Leibniz algebras $\g_1 \rightarrow \g_2$ is a morphism of $\kring$-modules that is compatible with the respective products.  The category of Leibniz algebras is denote $\leibalg$.  
\end{defn}

A Lie algebra is a Leibniz algebra such that the product is antisymmetric; in particular, there is a fully-faithful inclusion $
\liealg 
\hookrightarrow 
\leibalg
$ of the category $\liealg $ of Lie algebras.  This has left adjoint given by $\g \mapsto \glie$, where
$
\glie := \g / \langle x^2 \rangle_\kring
$, 
for $\langle x^2 \rangle_\kring$ the sub $\kring$-vector space generated by $x^2$,  $x \in \g$. 

\begin{nota}
The image of $y \in \g$ in $\glie$ is written $\overline{y}$.
\end{nota}

\begin{rem}
\label{rem:leibniz_kernel}
For a Leibniz algebra $\g$, the ideal $\langle x^2 \rangle_\kring \subset \g$ is sometimes called the Leibniz kernel of $\g$. It identifies as the image of the restriction of the product of $\g$ to the symmetric invariants $\Gamma^2 \g \subset \g^{\otimes 2}$, so that there is an exact sequence of $\kring$-vector spaces:
\[
\Gamma^2 \g \rightarrow \g \rightarrow \glie \rightarrow 0.
\]
\end{rem}

One has the right regular action of $\glie$ on $\g$:

\begin{prop}
\label{prop:right_reg_action}
\cite[Section 1.10]{LP} \cite[Section  3]{Feld}
For $\g$ a Leibniz algebra, the product $\g \otimes \g \rightarrow \g$,  induces a right action $\g \otimes \glie \rightarrow \g$, $x \otimes \overline{z} \mapsto xz$, of the Lie algebra $\glie$ on $\g$.
\end{prop}

\subsection{The Leibniz and Chevalley-Eilenberg complexes}

This section recalls the  Leibniz complex (or Loday complex) associated to a Leibniz algebra, together with its relationship with the Chevalley-Eilenberg complex of the associated Lie algebra.

\begin{defn}
\label{defn:leib_complex}
For $\g$ a Leibniz algebra, the Leibniz complex of $\g$ is $(T(s \g), \dleib)$, where $\dleib$ is the differential (of degree $-1$) given for $x_1, \ldots , x_n \in \g$ by 
\[
\dleib (sx_1 \otimes \ldots \otimes sx_n) 
= \sum_{1 \leq i< j \leq n}
(-1)^{j+1}
(sx_1 \otimes \ldots \otimes s (x_ix_j) \otimes \ldots \otimes \widehat{sx_j} \otimes \ldots \otimes sx_n) 
,
\]
where $\widehat{s x_j}$ indicates the omission of the term.

The Leibniz homology $HL_* (\g ; \kring)$ with trivial coefficients is the homology $H_* (T(s \g), \dleib)$.
\end{defn}

\begin{prop}
\label{prop:leib_ce}
 \cite{LP} \cite[Section 10.6.4]{MR1600246}
For $\g$ a Leibniz algebra, the symmetrization map $T(-) \twoheadrightarrow S^* (-)$ induces  natural morphisms of chain complexes 
\[
(T(s\g), \dleib) 
\twoheadrightarrow 
(T(s\glie), \dleib) 
\twoheadrightarrow 
\ce(\glie),
\]
where $\ce (\glie) : = (S^* (s \glie),d_{\ce})$ is the Chevalley-Eilenberg complex of $\glie$ and  the first map is induced by the canonical surjection $\g \twoheadrightarrow \glie$.

In homology, the composite induces the natural comparison map $HL_* (\g ;\kring) \rightarrow H_* (\glie; \kring)$ to the Lie algebra homology of $\glie$. 
\end{prop}

For current purposes, Proposition \ref{prop:leib_ce}  can be taken as giving the  definition of $\ce(\glie)$.

\subsection{Free Leibniz and free Lie algebras}

This Section recalls the homology of free Leibniz (respectively Lie) algebras. (This is used in Section \ref{subsect:free}.)

\begin{nota}
For $V$ a $\kring$-vector space, denote by
\begin{enumerate}
\item 
$\leibopd (V)$ the free Leibniz algebra on $V$; 
\item 
$\lieopd (V)$ the free Lie algebra on $V$. 
\end{enumerate}
\end{nota}

\begin{rem}
Both Leibniz and Lie algebras are encoded by operads and the above free functors arise from this framework \cite{LV} (see \cite[Section 13.2]{LV} for the Lie operad and \cite[Section 13.5]{LV} for the Leibniz operad).  In particular this implies that one can pass to the graded setting: for instance, the free Lie algebra $\lieopd (sV)$ on the suspension of a $\kring$-vector space is of importance here. 
\end{rem}

Since a Lie algebra is, in particular, a Leibniz algebra, there is a canonical morphism 
$
\leibopd (V) \rightarrow \lieopd(V)$ 
 of Leibniz algebras. More precisely, as a formal consequence of the adjunctions underlying the constructions, one has:

\begin{lem}
\label{lem:liezation_free_leib}
For $V$ a $\kring$-vector space, the canonical map $\leibopd (V) \rightarrow \lieopd(V)$ induces a natural 
 isomorphism 
$ 
(\leibopd(V))_{\mathrm{Lie}} \cong \lieopd (V)$
 of Lie algebras.
\end{lem}

The Leibniz homology of free Leibniz algebras and the Lie homology of free Lie algebras are well-understood:

\begin{prop}
\label{prop:hom_free_algebras}
For $V$ a $\kring$-vector space, there are canonical isomorphisms:
\begin{eqnarray*}
HL_n (\leibopd(V); \kring)
 &\cong &
 \left\{
\begin{array}{ll}
\kring & n=0 \\
V & n=1 \\
0 & n >1
\end{array}
\right.
\\
H_n (\lieopd(V); \kring) &\cong &
 \left\{
\begin{array}{ll}
\kring & n=0 \\
V & n=1 \\
0 & n >1.
\end{array}
\right.
\end{eqnarray*}

Moreover, the comparison map
$HL_* (\leibopd(V); \kring)
\rightarrow 
H_* (\lieopd (V); \kring)$ 
 associated to $\leibopd (V) \twoheadrightarrow \leibopd(V) _{\mathrm{Lie}} \cong \lieopd (V)$
is an isomorphism.
\end{prop}

\begin{proof}
One way to see this is within the general framework of algebraic operads \cite{LV}, as follows. Both the Leibniz operad and the Lie operad are binary quadratic Koszul operads; this  is equivalent to the given behaviour on free algebras (see \cite[Theorem 12.1.4]{LV}).
\end{proof}

\section{Pirashvili's subcomplex}
\label{sect:subcomplex}

This Section gives a presentation of the construction of Pirashvili's subcomplex $(\lieopd (s\g), \dleib)$ of the Leibniz complex $(T(s \g), \dleib)$ of a Leibniz algebra $\g$, that was introduced in \cite{2019arXiv190400121P}. The method used will be generalized in Section \ref{sect:prim} to study the primitive filtration of $T(s \g)$.

\subsection{Reconstructing the Leibniz differential}

For $V$ a $\kring$-vector space, the tensor algebra $T(sV)$ has the structure of a primitively-generated Hopf algebra (working in the graded setting with Koszul signs). The coproduct $\Delta : T(sV) \rightarrow T(sV) \otimes T(sV)$ is determined by the fact that the subspace $sV$ lies in the primitives; it identifies explicitly as the shuffle coproduct (with Koszul signs). 

The  Cartier-Milnor-Moore theorem (see \cite{MM} and \cite[Theorem B.4.5]{MR258031}, for example) gives:

\begin{prop}
\label{prop:CMM}
The primitives of the Hopf algebra $T(sV)$ identify as $\lieopd (sV) \subset T(sV)$, the free Lie algebra on $sV$. Moreover, the canonical map 
$
U (\lieopd (sV) ) \stackrel{\cong}{\rightarrow} T(sV)
$ 
from the universal enveloping algebra of $\lieopd (sV)$ is an isomorphism of Hopf algebras.
\end{prop}

\begin{rem}
The free Lie algebra functor $\lieopd (-)$ has a weight grading, so that $\lieopd (-) = \bigoplus_{n \geq 1} \lieopd_n (-)$,  with $\lieopd_n (-) \subset T^n (-)$. (This follows, for instance, from the fact that Lie algebras are encoded by an operad.)

Consider the graded situation of Proposition \ref{prop:CMM}. In small degrees, $\lieopd_1 (sV) = sV = T^1 (sV)$ and $\lieopd_2 (sV) \cong s^2 \Gamma^2 (V) \subset s^2 T^2(V) \cong T^2(sV)$.
 The homological suspension leads to the  relation $[sx , sy] = [sy,sx]$ for $x, y \in V$, which accounts for the appearance of the functor $\Gamma^2$ (as opposed to the second exterior power $\Lambda^2$ which arises when considering $\lieopd_2 (V)$).
\end{rem}

\begin{nota}
\label{nota:ld}
Let $\ld : T(sV) \rightarrow sV \otimes T(sV)$ be the component of the coproduct $\Delta : T(sV) \rightarrow T(sV) \otimes T(sV)$ mapping to terms linear in the first tensor factor. Explicitly, on $T^n (sV)$, for $1\leq n \in \nat$:
\[
\ld (sx_1 \otimes \ldots \otimes sx_n) 
= 
\sum_{i=1}^n
(-1)^{i+1} 
sx_i \otimes \big (
sx_1 \otimes \ldots \otimes \widehat{sx_i} \otimes \ldots \otimes sx_n
\big) \in sV \otimes T^{n-1} (sV).
\]
\end{nota}

The following Lemma is a direct consequence of Proposition \ref{prop:CMM}, since $\lieopd (sV)$ corresponds to the primitives of $T(sV)$:

\begin{lem}
\label{lem:ld_lie}
The restriction of $\ld$ to $\lieopd_n (sV)$ is trivial for $1<n \in \nat$. On the linear terms $sV = \lieopd_1 (sV)  \subset \lieopd (sV)$ it identifies as the inclusion $sV \cong sV \otimes \kring \subset sV \otimes T(sV)$.
\end{lem}

Consider a Leibniz algebra $\g$ and its associated Lie algebra $\glie$, so that $\glie$ acts on $\g$ via the right regular action (see Proposition \ref{prop:right_reg_action}). Hence $\glie$ acts on the right on $s\g$ via $ (sx) \otimes \overline{z} \mapsto s (xz)$ and this action extends for $n\in \nat$ to a right action of $\glie$ on $T^n (s\g) = (s\g)^{\otimes n}$, using the monoidal structure of Lemma \ref{lem:lie_symm_mon}.

This defines the right action of $\glie$ on $T(s\g)$, and hence that of $\g$, via the composite 
$T(s\g) \otimes \g \twoheadrightarrow T (s\g) \otimes \glie \rightarrow T(s\g)$, where the first map is induced by the canonical surjection $\g \twoheadrightarrow \glie$.  

\begin{nota}
\label{nota:rr_action_T}
Write $X \otimes y \mapsto X \cdot y$ for the right action $T (s\g) \otimes \g \rightarrow T (s\g)$, where $X \in T(s\g) $ and $y \in \g$. 
\end{nota}

\begin{lem}
\label{lem:glie_action}
\ 
\begin{enumerate}
\item 
The right regular action of $\glie$ on $\g$ equips $T(s\g)$ with the structure of a cocommutative Hopf algebra in the category of right $\glie$-modules.
\item 
$\lieopd (s\g) \subset T(s\g) $ is stable under the $\glie$ action, hence defines a sub $\glie$-module.
\item 
The Lie bracket $[-,-] \ : \  \lieopd (s\g) \otimes \lieopd (s\g) \rightarrow \lieopd (s\g)$ is a morphism of $\glie$-modules, where the domain is given the tensor product $\glie$-module structure of Lemma \ref{lem:lie_symm_mon}.
\end{enumerate}
\end{lem}

\begin{proof}
The first statement is an immediate consequence of Proposition \ref{prop:right_reg_action} (using the symmetric monoidal structure given by Lemma \ref{lem:lie_symm_mon}).  

The second statement then follows, since $\lieopd(s \g) $ identifies as the primitives of the Hopf algebra structure on $T(s \g)$ by Proposition \ref{prop:CMM} and these primitives are constructed as an equalizer of a diagram that is defined in right $\glie$-modules.

The final statement follows by again using the symmetric monoidal structure given by Lemma \ref{lem:lie_symm_mon}. 
\end{proof}

The Lie action can be used to (re)construct the differential of the Leibniz complex $(T(s\g), \dleib)$ (cf. Definition \ref{defn:leib_complex}) as follows.

\begin{nota}
For $n \in \nat$, let $\dlie_{n,1}$ be the $\kring$-linear map (of degree $-1$)
\[
(s \g)^{\otimes n+1} 
\cong 
(s \g)^{\otimes n} \otimes s\g 
\cong s \big ( (s \g)^{\otimes n} \otimes \g \big)
\twoheadrightarrow 
  s \big ( (s \g)^{\otimes n} \otimes \glie \big)
  \rightarrow 
  s(s \g)^{\otimes n} 
  \stackrel{\cong}{\rightarrow}
  (s \g)^{\otimes n} 
  \]
  where the penultimate map is given by  the $\glie$-action, the final map is the canonical isomorphism of degree $-1$  and the second isomorphism introduces the sign $(-1)^n$ via the Koszul rule.
\end{nota}

\begin{rem}
\ 
\begin{enumerate}
\item 
Observe that  $\dlie_{0,1}$ is zero, since the $\glie$-action on $\kring$ is trivial. The map $\dlie_{1,1} : s\g \otimes s\g \rightarrow s\g$ is given by  $sx \otimes sy \mapsto - s(xy)$, for $x, y \in \g$.
\item 
Using Notation \ref{nota:rr_action_T}, the action of $\dlie_{n,1}$ for $X \in T^n (s\g)$ and $y \in \g$ is given by 
\begin{eqnarray}
\label{eqn:dlie}
\dlie_{n,1} (X \otimes sy) = (-1)^{n} X \cdot y,
\end{eqnarray}
which highlights the relation between $\dlie_{n,1}$ and the right action of $\g$ on $T(s\g)$.
\end{enumerate}
\end{rem}

\begin{prop}
\label{prop:leib_diff}
The Leibniz differential $(\dleib)_{n+1} : (s \g)^{\otimes n+1} \rightarrow (s\g)^{\otimes n}$, $n \in \nat$,
 satisfies $(\dleib)_0=0$ and the recursive relation 
  \[
  (\dleib)_{n+1} := ((\dleib)_n \otimes \id) + \dlie_{n,1}.
  \]
 Explicitly, for $X \in (s\g)^{\otimes n} \subset T(s\g)$ and $y \in \g$, 
\[
\dleib (X \otimes sy)  = (\dleib X) \otimes sy + \dlie_{n,1} (X \otimes sy).
\] 
\end{prop}

\begin{proof}
This is proved by a direct verification using the explicit form of the Leibniz differential  given in Definition \ref{defn:leib_complex}. The key point is that the sign appearing for the term indexed by $i<j$ only depends upon $j$.
\end{proof}

\begin{rem}
\label{rem:dleib_glie-linear}
Using  the relation $\dleib^2=0$, Proposition \ref{prop:leib_diff} gives the  
 anti-commutation relation between $\dleib$ and $\dlie$:
\[
 \dleib \dlie_{n,1} (X \otimes sy) + \dlie_{n-1, 1} ((\dleib X) \otimes sy) =0,
\]
This is equivalent to the fact that the Leibniz differential $\dleib$ is $\glie$-linear for the right regular action. 
\end{rem}

The map $\dlie_{n,1}$ extends:

\begin{defn}
\label{defn:dlie_nt}
For $n, t \in \nat$, let $\dlie_{n,t} : (s\g)^{\otimes n +t} 
\rightarrow (s\g)^{\otimes n + t-1}$ be the composite (of degree $-1$): 
\[
(s\g)^{\otimes n +t}
\cong 
(s\g)^{\otimes n} \otimes (s\g)^{\otimes t} 
\stackrel{\id \otimes \ld}{\longrightarrow}
(s\g)^{\otimes n} \otimes s \g \otimes (s\g)^{\otimes t-1}
\stackrel{\dlie_{n,1} \otimes \id}{\longrightarrow}
 (s\g)^{\otimes n} \otimes (s\g)^{\otimes t-1}
 \cong (s\g)^{\otimes n + t-1}
\]
if $t>0$ and zero for $t=0$.
 \end{defn}

Proposition \ref{prop:leib_diff} generalizes to:

\begin{prop}
\label{prop:diff_dlie}
For $n,t \in \nat$, the Leibniz differential 
$
(\dleib)_{n+t} : 
(s\g)^{\otimes n+t } 
\cong 
(s\g)^{\otimes n} \otimes (s\g)^{\otimes t} 
\rightarrow (s\g)^{\otimes n + t-1}
$  
identifies as 
\[
(\dleib)_{n+t} = (\dleib)_n \otimes \id \ + \ (-1)^n \id \otimes (\dleib)_t \ + \ \dlie_{n,t}.
\]
\end{prop}

\begin{proof}
For $t =0$ the statement is vacuous and, for $t=1$, is a restatement of Proposition \ref{prop:leib_diff}.
 The general case is proved similarly to Proposition \ref{prop:leib_diff}.
\end{proof}

\begin{cor}
\label{cor:diff_dlie_prim}
For $n,t \in \nat$, suppose that $X \in(s\g)^{\otimes n}$ and  $Y \in \lieopd_t(s\g) \subset (s\g)^{\otimes t}$ with $t >1$. Then 
\[
 \dleib (X \otimes Y) 
 = 
 \dleib X \otimes Y + (-1)^n X \otimes \dleib Y 
 \]
 \end{cor}
 
 \begin{proof}
 The hypothesis that $t >1$ and $Y$ is primitive implies that $\ld Y = 0$ by Lemma \ref{lem:ld_lie}, hence $\dlie_{n,t} (X \otimes Y)=0$. Thus the result follows from Proposition \ref{prop:diff_dlie}.
 \end{proof}

\subsection{The subcomplex}

Corollary \ref{cor:diff_dlie_prim} leads to an alternative proof of \cite[Proposition 2.1]{2019arXiv190400121P}:

\begin{prop}
\label{prop:lie_stable}
For $\g$ a Leibniz algebra, $\lieopd (s\g) \subset T(s\g)$ is stable under the Leibniz differential. 

Moreover, for $X \in \lieopd_n(s \g)$ with $n \geq 2$ and $y \in  \g$:
$$ \dleib [X, sy] 
= 
 [\dleib X, sy] 
+ 
\dlie_{n,1} (X \otimes sy).
$$
\end{prop}

\begin{proof}
Since $\lieopd_1 (s\g) = s\g$ and the differential $\dleib$ vanishes on $s\g$, the result holds for 
$\dleib$ restricted to $\lieopd_n (s \g)$ for $n \leq 2$.  Starting from this, the result is proved by increasing induction upon $n$. 

For $n \geq 1$, the Lie bracket induces a surjection 
\[
\lieopd_n (s\g) \otimes s \g \stackrel{[-,-]}{\twoheadrightarrow} \lieopd_{n+1} (s\g).
\]
Hence, for the inductive step (with $n \geq 2$), it suffices to show that, given $X \in \lieopd_n(s \g)$ and $y \in  \g$, the commutator
\[
[X, sy] = X \otimes sy - (-1)^n sy \otimes X \in \lieopd_{n+1} (s\g)
 \] 
has differential $\dleib [X,sy]$ that lies in $\lieopd_n (s\g) \subset T^n (s\g)$ and that this is given by the 
explicit expression of the statement.

By Proposition \ref{prop:diff_dlie} and Corollary \ref{cor:diff_dlie_prim}
\begin{eqnarray*}
\dleib [X, sy] 
&=&
\dleib X \otimes sy 
+ \dlie_{n,1} (X \otimes sy) 
- \Big( 
(-1)^{n-1} sy \otimes \dleib X 
+ 
\dlie_{1,n} (sy \otimes X) 
\Big)
\\
&=&
[\dleib X , sy]  
+ \dlie_{n,1} (X \otimes sy), 
\end{eqnarray*}
where the second equality uses $\dlie_{1,n} (sy \otimes X)=0$, as in Corollary \ref{cor:diff_dlie_prim}, since $n \geq 2$.

By the inductive hypothesis, $\dleib X\in \lieopd_{n-1} (s \g)$, hence $[\dleib X , sy]  \in \lieopd_{n} (s \g)$.  Lemma \ref{lem:glie_action} implies that $\dlie_{n,1} (X \otimes sy) \in \lieopd_n (s\g)$, since this is induced by the right regular action. This completes the inductive step and hence the proof.
\end{proof}

As pointed out by the referee, this gives the following useful relationship between the right action of $\g$ on $T(s\g)$ and the Leibniz differential, using Notation \ref{nota:rr_action_T}:

\begin{cor}
\label{cor:rr_dleib}
For $n \geq 2$, $X \in \lieopd_n (s \g)$ and $y \in \g$, one has:
\[
X \cdot y = (-1)^{n} \big( \dleib [X, sy] - [\dleib X, sy] \big).
\]
\end{cor}

\begin{proof}
This is an immediate consequence of the explicit identification given in Proposition \ref{prop:lie_stable},  using the equality given in equation (\ref{eqn:dlie}).
\end{proof}

The following gives some initial information on the homology of $(\lieopd(s\g), \dleib)$:

\begin{prop}
\label{prop:Lie_hom_glie_action}
For $\g$ a Leibniz algebra,  $H_* (\lieopd(s\g), \dleib)$ takes values in $\glie$-modules.
\begin{enumerate}
\item 
There are $\glie$-module isomorphisms
\[
H_* (\lieopd(s\g), \dleib) \cong 
\left\{
\begin{array}{ll}
0 & *=0 \\
(s \glie)_1 \cong \glie & *=1 
\end{array}
\right.
\]
\item 
For $*\geq 2$,  $H_* (\lieopd(s\g), \dleib)$ has trivial $\glie$-action. 
\end{enumerate}
\end{prop}

\begin{proof}
The $\glie$-action is induced by the $\glie$-action on $T(s\g)$. Explicitly, if $X \in \lieopd_n (s\g)$ is a cycle and $sx$ in $s\g$ (for $x \in \g$, with associated element $\overline{x} \in \glie$),  the class $[X]\overline{x}$ is represented by the cycle $(-1)^n \dlie_{n,1} (X \otimes sx)$ (cf. equation (\ref{eqn:dlie})).

 The vanishing of homology in degree zero is immediate and the degree $1$ homology  is given by \cite[Lemma 2.3]{2019arXiv190400121P}. This can be seen explicitly as follows: in homological degrees $\leq 2$, the complex has the form 
 \[
 \lieopd_2 (s \g) \cong s^2 \Gamma ^2 \g 
 \rightarrow 
 \lieopd_1 (s \g) = s\g 
 \rightarrow 
 0 = \lieopd_0 (s \g)
 \]
 and, up to sign, the differential $ \lieopd_2 (s \g) \rightarrow  \lieopd_1 (s \g)$ is induced by the restriction of the Leibniz product to $\Gamma^2 \g \subset \g^{\otimes 2}$. This has cokernel $\glie$ (see Remark \ref{rem:leibniz_kernel}) and  gives the isomorphism of $\glie$-modules.

It remains to show that the $\glie$-action is trivial in higher homological degree. Consider $X$, $sx$ as above, with $n \geq 2$. Thus $[X, sx] \in \lieopd_{n+1}(s\g)$ and, since $X$ is a cycle, by Corollary \ref{cor:rr_dleib} one has 
\[
X \cdot x = (-1)^n \dleib [X,sx] 
\]
and hence, in homology, $[X]\overline{x}=0$, as required.
\end{proof}

\section{The primitive filtration and its spectral sequence}
\label{sect:prim}

This Section gives the proof of a conjecture due to Pirashvili and Loday \cite{2019arXiv190400121P}, namely that the primitive filtration gives a filtration (as complexes) of the Leibniz complex.  This leads to the  spectral sequence introduced in Section \ref{subsect:ss}.

\subsection{The primitive filtration}

Recall that, for $\h$ a graded Lie algebra, the universal enveloping algebra $U(\h)$ is the quotient of the tensor algebra $T(\h)$ by the two-sided ideal generated by $(x \otimes y - (-1)^{|x||y|} y \otimes x) - [x,y]$ for $x, y\in \h$, where $[x,y]$ is the commutator in $\h$. 

One has the primitive filtration of $U(\h)$ (cf. \cite[Definition 5.10]{MM}):

\begin{defn}
\label{defn:filt_Uh}
For $n \in \nat$, let $\filt_n U(\h)$ denote the image of the composite
$
\bigoplus_{i \leq n} T^n (\h) 
\subset T(\h)
\twoheadrightarrow 
U(\h).
$
\end{defn}

By construction, this is an exhaustive increasing filtration of $U(\h)$; it  is  compatible with the product of $U(\h)$, i.e.,  the product restricts for $m, n \in \nat$ to 
\[
\filt_m U(\h) 
\otimes 
\filt_n U(\h) 
\rightarrow 
\filt_{m+n} U (\h)
\]
and the associated graded is the free graded commutative algebra on $\h$, 
$ 
\mathsf{gr} \ \filt_\bullet U(\h) 
\cong 
S^* (\h),
$
 equipped with both the length grading and the homological grading inherited from that  of $\h$.  Explicitly, for $n \in \nat$:
\[
\filt_n U(\h) / \filt_{n-1} U(\h) 
\cong 
S^n (\h),
\]
where the formation of  $S^n(-)$ takes into account the Koszul signs arising from the homological degree.

In the case $\h = \lieopd (sV)$, the universal enveloping algebra $U(\lieopd (sV))$ is naturally isomorphic to $T(sV)$ (see Proposition \ref{prop:CMM}). In particular, the above gives a filtration of the tensor algebra $T(sV)$ with $\filt_n T(sV)$ the image of the composite
\[
\bigoplus_{i \leq n} \lieopd(sV) ^{\otimes i} 
\subset T(\lieopd (sV))
\twoheadrightarrow 
U(\lieopd (sV)) \cong T(sV).
\]
Thus, $\filt_0 T(sV) = \kring$ and $\filt_1 T(sV) = \kring \oplus \lieopd (sV)$.

\subsection{Loday's conjecture}

For $\g$ a Leibniz algebra, one has the Hopf algebra $T(s\g)$ and the underlying graded vector space is equipped with the Leibniz differential $\dleib$. 
 Loday suggested (see \cite[Conjecture 2]{2019arXiv190400121P}) that the stability of $\lieopd (s\g)$ under the Leibniz differential (see Proposition \ref{prop:lie_stable}) should generalize to the appropriate filtration of $T(s\g)$. The following  
establishes this conjecture:

\begin{thm}
\label{thm:prim_filt_stable}
For $\g$ a Leibniz algebra and $n \in \nat$, $\filt_n T(s \g)$ is a subcomplex of the Leibniz complex $(T(s\g),\dleib)$. This defines an increasing, exhaustive filtration 
\[
\filt_{0} T(s\g) = \kring \subset 
\ldots \subset \filt_n T(s\g) \subset \filt_{n+1} T(s\g) \subset \ldots \subset T(s\g).
\]

Moreover, there is a natural isomorphism of complexes:
\[
\filt_n T(s\g)/ \filt_{n-1} T(s\g)
\cong 
S^n (\lieopd (s\g), \dleib).
\]
\end{thm}

\begin{proof}
First consider the stability of the subcomplexes. The case $n=0$ is obvious and the case $n=1$ is given by Proposition \ref{prop:lie_stable}.

The result for $n \geq 2$ is proved by increasing induction upon $n$. For the inductive step, it suffices to prove that, if $X_i \in \lieopd (s\g)$ for $1 \leq i \leq n$, then 
\[
\dleib (X_1 \otimes X_2 \otimes \ldots \otimes X_n) \in \filt_n T(s \g).
\]

Suppose that $X_i \in \lieopd_{n_i}(s\g) \subset (s\g)^{\otimes n_i}$ with $n_i \geq 1$ for each $i$. Then a straightforward induction upon $n$ using Corollary \ref{cor:diff_dlie_prim} gives that  
\begin{eqnarray}
\label{eqn:diff_gen}
\dleib (X_1 \otimes X_2 \otimes \ldots \otimes X_n)
&=& 
\sum_{i=1}^n 
(-1)^{n(i)} X_1 \otimes \ldots \otimes  \dleib  X_i \otimes \ldots \otimes X_n
\\
&&
+
\sum_{\{ i | n_i =1 \}} 
\dlie_{n(i),1} \big( (\bigotimes_{j<i} X_j) \otimes X_i \big) \otimes (\bigotimes_{k>i}X_k)
\nonumber
\end{eqnarray}
where $n(i) := \sum_{j=1}^{i-1} n_j$. Here  $\dleib X_i \in \lieopd( s\g )$, by Proposition \ref{prop:lie_stable}, and  $\dlie_{n(i),1} \big( (\bigotimes_{j<i} X_j) \otimes X_i \big)$ lies in $\lieopd (s\g) ^{\otimes (i-1)}$, by Lemma \ref{lem:glie_action}. 

Hence $\dleib (X_1 \otimes X_2 \otimes \ldots \otimes X_n)$ lies in the image of 
\[
\lieopd (s\g) ^{\otimes n}  \oplus \lieopd (s\g)^{\otimes n-1} \subset T(\lieopd(s\g) ) \rightarrow T(s\g), 
\]
in particular it lies in $\filt_n T (s \g)$. This establishes that $\filt_n T(s\g)$ is a subcomplex of $(T(s\g), \dleib)$. 

The above also implies that there is a surjection of {\em complexes} 
\[
(\lieopd (s\g), \dleib )^{\otimes n} 
\twoheadrightarrow 
\filt_n T(s \g) / \filt_{n-1} T (s\g),
\]
since the second term of equation (\ref{eqn:diff_gen}) lies in filtration $n-1$. It follows that the canonical factorization 
\[
S^n (\lieopd (s\g)) \stackrel{\cong}{\rightarrow} \filt_n T(s \g) / \filt_{n-1} T (s\g)
\] 
is an isomorphism of complexes, as required.
\end{proof}

\subsection{The associated spectral sequence}
\label{subsect:ss}

The filtration given by Theorem \ref{thm:prim_filt_stable} yields a spectral sequence for calculating the Leibniz homology of $\g$ with trivial coefficients, starting from knowledge of the homology of the subcomplex $(\lieopd (s\g), \dleib)$:

\begin{thm}
\label{thm:ss}
For $\g$ a Leibniz algebra, there is a convergent first quadrant homological spectral sequence 
\[
E^1_{p,q} 
= 
H_{p+q} (\filt_p T(s\g)/ \filt_{p-1} T(s\g)) 
\Rightarrow 
H_{p+q} (T(s\g), \dleib ) = HL_{p+q} (\g; \kring).
\]
Moreover, $H_*(\filt_p T(s\g)/ \filt_{p-1} T(s\g)) \cong S^p (H_* (\lieopd (s \g), \dleib))$. 
In particular, 
\begin{eqnarray*}
E^1 _{0,q} 
&= &
\left \{ 
\begin{array}{ll}
\kring & q=0 \\
0 & q>0.
\end{array}
\right.
\\
E^1 _{1,q} & \cong & H_{q+1} (\lieopd (s\g), \dleib). 
\end{eqnarray*}

\begin{enumerate}
\item 
  $(E^1_{p,0}, d^1)$  identifies with the Chevalley-Eilenberg complex $\ce(\glie)$;
  \item 
   the $q=0$ edge homomorphism of the spectral sequence identifies with the comparison map
 $
HL_* (\g; \kring) 
\rightarrow 
H_* (\glie; \kring).
$ 
\end{enumerate}
\end{thm}

\begin{proof}
This is the spectral sequence associated to the filtration of the Leibniz complex $(T(s\g), \dleib)$ given by Theorem \ref{thm:prim_filt_stable}. Since $\kring$ is a field of characteristic zero, it follows that the $E^1_{p,q}$-page, which  is given by the homology  $H_* (S^\bullet (\lieopd (s\g), \dleib))$, is isomorphic to $S^\bullet (H_* (\lieopd (s \g), \dleib))$,  imposing the usual grading for the spectral sequence.

It remains to identify the row $q=0$ and its differential $d^1$. One has $E^1_{p,0}= S^p (H_1 (\lieopd (s \g), \dleib))$ and $H_1 (\lieopd (s\g), \dleib) \cong (s\glie)_1 $ by Proposition \ref{prop:Lie_hom_glie_action}. Hence $ E^1_{p,0}$ identifies as the degree $p$ part of $S^* (s\glie)$. The differential $d^1$ is induced by the Leibniz differential, hence the identification with the Chevalley-Eilenberg complex of $\glie$ follows from Proposition \ref{prop:leib_ce}, which also gives the identification of the $q=0$ edge homomorphism. 
\end{proof}

The explicit identification of the $E^1$-page in terms of $H_* (\lieopd (s \g), \dleib)$ leads to vanishing zones in the spectral sequence. For instance, one has the  following (which is applied in Theorem \ref{thm:comparison} below).

\begin{cor}
\label{cor:vanishing_zone}
For $\g$ a Leibniz algebra such that $H_* (\lieopd (s \g), \dleib)=0$ for $1 < * <N$, in the spectral sequence of Theorem \ref{thm:ss}, 
\[
E^1_{p,q} =0 \mbox{ for $0 < q < N-1$.}
\]
\end{cor}

\begin{proof}
This is a direct consequence of the identification of the $E^1$-term given in Theorem \ref{thm:ss},  as follows. The case $p=0$ is immediate.  For $p>0$, $E^1_{p,q} = S^p (H_* (\lieopd (s \g), \dleib))_{p+q}$; the first potentially non-zero term in positive $q$-degree is given by the image of $H_N   (\lieopd (s \g), \dleib) \otimes S^{p-1} (H_1 (\lieopd (s \g), \dleib)) \rightarrow S^p (H_* (\lieopd (s \g), \dleib))$ induced by the product.  This is in $q$-degree $N-1$.
\end{proof}

It is not difficult to analyse the differential $d^1$ in the spectral sequence. The following is sufficient for current purposes:

\begin{prop}
\label{prop:d1}
The differential $d^1 : E^1_{2, q} \rightarrow E^1 _{1, q}$ is zero for $q>0$. 
\end{prop}

\begin{proof}
By Theorem \ref{thm:ss}, $E^1 _{1, q}$ identifies as $H_{q+1} (\lieopd (s \g), \dleib)$, in particular $E^1 _{1,0} = H_{1} (\lieopd (s \g), \dleib) \cong \glie$, by Proposition  \ref{prop:Lie_hom_glie_action}.  Likewise  $E^2_{2, q} = S^2 (H_* (\lieopd (s \g), \dleib))_{q+2}$.

The differential $d^1$ is analysed by using the methods employed in the proof of Theorem \ref{thm:prim_filt_stable}, as follows. Consider non-zero homology classes $[X_1]$, $[X_2]$ in $H_* (\lieopd (s \g), \dleib)$ and their product in $S^2 (H_* (\lieopd (s \g), \dleib))$. Thus $X_i \in \lieopd_{n_i} (s\g)$ are cycles, where $n_1, n_2>0$ (by the non-triviality hypothesis) and, without loss of generality, we may assume that $n_1 \leq n_2$. The $q$-degree of the product class in $S^2 (H_* (\lieopd (s \g), \dleib))$ is $n_1 +n_2 -2$; this implies that $ n_2>1$  if $q>0$.

The element $X_1 \otimes X_2$ lies in $\filt_2 T(s \g)$ and defines a cycle in $\filt_2 T(s \g)/\filt_1 T(s\g)$ that represents the product class in $S^2 (H_* (\lieopd (s \g), \dleib))$. Since $X_1$, $X_2$ are cycles, Corollary \ref{cor:diff_dlie_prim} gives that $X_1 \otimes X_2$ is a cycle in $\filt_2 T(s \g)$, since $n_2>1$. By the construction of the differential $d^1$ in a spectral sequence associated to a filtration, this implies that $d^1=0$.
\end{proof}

\section{Applications of the spectral sequence}
\label{sect:appli}

The spectral sequence of Theorem \ref{thm:ss} is applied to relate the behaviour of the comparison map $HL_*(\g; \kring) \rightarrow H_*(\glie; \kring)$ to vanishing zones of $H_* (\lieopd (s \leibopd (V)), \dleib)$.

\subsection{The case of a free Leibniz algebra}
\label{subsect:free}

Consider $\g := \leibopd (V)$, the free Leibniz algebra on a $\kring$-vector space, so that  one has the Leibniz complex $(T(s \leibopd (V)), \dleib)$ and its subcomplex $(\lieopd (s \leibopd (V)), \dleib)$.  
In \cite[Conjecture 1]{2019arXiv190400121P}, Pirashvili conjectured that the homology of the latter is concentrated in degree $1$. The spectral sequence of Theorem \ref{thm:ss}  gives a new proof of this conjecture:

\begin{thm}
\label{thm:liezation}
For $V$ a $\kring$-vector space, 
\[
H_* (\lieopd (s \leibopd (V)), \dleib)\cong 
\left\{ 
\begin{array}{ll}
\lieopd (V) & *=1 \\
0 & \mbox{otherwise.}
\end{array}
\right.
\]
\end{thm}

\begin{proof}
Use the spectral sequence of Theorem \ref{thm:ss} for $\g = \leibopd (V)$; this calculates the Leibniz homology $HL_*(\leibopd(V); \kring)$ starting from $H_* (\lieopd (s \leibopd (V)), \dleib)$.
In particular, $E^1_{0,*}=0$ for $*>0$ and $E^1_{1,*}$ identifies as $H_{*+1} (\lieopd (s \leibopd (V)), \dleib)$. 

Here, the information given by Proposition \ref{prop:hom_free_algebras} on the homology of the respective free algebras  implies the following:
\begin{enumerate}
\item 
$E^2_{0,0} = \kring$, $E^2_{1,0} = V$ and $E^2_{p,0}=0$ for $p>1$, since this is the Lie homology of the free Lie algebra $\lieopd (V)$; 
\item 
$E^\infty_{0,0} = \kring$, $E^\infty_{1,0} = V$ and $E^\infty _{p,q}=0$ otherwise, since this is an associated graded of the Leibniz homology of $\leibopd (V)$ and the $q=0$ edge homomorphism is an isomorphism (using Proposition \ref{prop:hom_free_algebras}).
\end{enumerate}

Suppose that $H_{*+1} (\lieopd (s \leibopd(V)), \dleib)$ is non-zero for some $*>0$ and take $0< \alpha < \infty$ minimal such that $ 
H_{\alpha +1} (\lieopd (s \leibopd(V)), \dleib)\neq 0$. Thus, Corollary \ref{cor:vanishing_zone} gives $E^1 _{p, q} =0$ for $1 \leq q \leq \alpha -1$ and, by construction, $E^1_{1,\alpha} \neq 0$.  

Proposition \ref{prop:d1} gives that $d^1 : E^1_{2, \alpha} \rightarrow E^1 _{1, \alpha}$ is zero (using that $\alpha>1$), so that $E^2 _{1, \alpha} \neq 0$. This is a contradiction, since the
vanishing $E^2_{p,0}=0$ for $p>1$ and the choice of $\alpha$ ensures that $E^2 _{1, \alpha}$ survives the spectral sequence. The result follows.
\end{proof}

\begin{rem}
This result was first  proved by Mostovoy (see \cite[Proposition 13]{2019arXiv191003754M}) by different methods, based on his interpretation of Leibniz algebras in terms of differential graded Lie algebras.
\end{rem}

\subsection{Connectivity of the comparison map}

The method used in the proof of Theorem \ref{thm:liezation} generalizes to give the following:

\begin{thm}
\label{thm:comparison}
Let $\g$ be a Leibniz algebra. 
\begin{enumerate}
\item 
The comparison map $HL_* (\g; \kring) \rightarrow H_* (\glie; \kring)$ is an isomorphism if and only if $H_* (\lieopd (s\g), \dleib)= 0$ for all $*>1$. 
\item 
More generally, the following conditions are equivalent:
\begin{enumerate}
\item 
 $H_* (\lieopd (s\g), \dleib))= 0$ for $1 < * < N$  and $H_N (\lieopd(s \g), \dleib) \neq 0$ for some $N\in \nat$. 
 \item 
The comparison map $HL_* (\g; \kring) \rightarrow H_* (\glie; \kring)$ is an isomorphism for $*<N$ and there is an exact sequence
\[
HL_{N+1} (\glie; \kring)
\rightarrow 
H_{N+1} (\glie; \kring) 
\rightarrow 
H_N (\lieopd(s \g), \dleib)
\rightarrow 
HL_N(\g; \kring) 
\rightarrow 
H_N(\glie; \kring) 
\rightarrow 
0
\]
and at least one of the following holds:
\begin{enumerate}
\item 
 $HL_N(\g; \kring) 
\rightarrow 
H_N(\glie; \kring) $ is not an isomorphism;
\item 
$HL_{N+1} (\glie; \kring)
\rightarrow 
H_{N+1} (\glie; \kring) $ is not surjective.
\end{enumerate}
\end{enumerate}
\end{enumerate}
\end{thm}

\begin{proof}
The first case is proved as for Theorem \ref{thm:liezation} by using the hypothesis that the comparison map is an isomorphism to deduce that the differentials $d^r$ on the row $q=0$ are zero for $r\geq 2$. 

The second case refines this argument using Corollary \ref{cor:vanishing_zone}. The first possible non-trivial differential originating on $q=0$ is $d^{N}: E^{N}_{N+1,0} \rightarrow E^N_{1,N-1}$, where $E^N_{N+1,0} \cong E^2_{N+1,0}$ and $E^N_{1, N-1} \cong E^1 _{N-1,0} \neq 0$. This gives the five term exact sequence and the equivalence of the conditions.
\end{proof}

\part{The natural $L_\infty$-structure on Pirashvili's complex}

\section{Constructing  $L_\infty$-structures}
\label{sect:lift}

This Section reviews $L_\infty$-structures and a version of the homotopy transfer theorem, which allows a `$3$-stub' of an $L_\infty$-structure to be extended to a full structure (see Proposition \ref{prop:extend}). A refinement of this argument is employed in the proof of Theorem \ref{thm:main}.

\subsection{$L_\infty$-structures}

Let $W$ be a  graded $\kring$-vector space. The symmetric algebra $S^* (sW)$ is formed using the symmetric monoidal structure on graded $\kring$-vector spaces with Koszul signs, hence is graded commutative. It has the structure of a graded bicommutative Hopf algebra primitively generated by $sW$; in particular one has the coproduct 
\[
\Delta : S^* (sW) \rightarrow S^* (sW) \otimes S^* (sW).
\]

\begin{defn}
\label{defn:linf}
(Cf. \cite[Chapter 7]{MR2931635} for example.)
An $L_\infty$-structure upon $W$ is $(S^* (sW), d)$, where $d : S^* (sW) \rightarrow  S^* (sW)$  is a coderivation of degree $-1$ and a differential (i.e., $d^2=0$). 
\end{defn}

The coderivation property implies that $d$ is determined by the composite $\dbar$:
\[
S^* (sW) \stackrel{d} {\rightarrow}  S^* (sW) \twoheadrightarrow sW,
\]
where the second map is the projection $S^* (sW) \twoheadrightarrow sW = S^1 (sW)$ onto the linear part. The full coderivation $d$ is recovered as the composite:
\[
S^* (sW) 
\stackrel{\Delta} {\rightarrow} 
S^* (sW) \otimes S^* (sW) 
\stackrel{\dbar \otimes \id}{\rightarrow} 
 sW \otimes S^* (sW) 
\stackrel{\mu}{\rightarrow} 
S^* (sW),
\]
where $\mu$ denotes the product of $S^* (sW)$.

\begin{rem}
\label{rem:coderiv_extn}
The above extension is given by the natural  $\kring$-linear map 
$$ 
\hom_\kring (S^* (sW), sW ) 
\rightarrow 
\hom_\kring (S^*(sW), S^* (sW))
$$  
that is constructed using the natural transformations $\Delta$ and $\mu$. It has a natural retract induced by the projection $S^* (sW) \twoheadrightarrow sW$.
\end{rem}

The following is standard:

\begin{lem}
\label{lem:diff_coderiv}
For $d : S^* (sW) 
\stackrel{d}{\rightarrow}
S^* (sW)$ a coderivation, $d$ is a differential if and only if the following composite vanishes:
\[
S^* (sW) 
\stackrel{d}{\rightarrow}
 S^* (sW) 
\stackrel{\dbar}{\rightarrow} 
sW.
\]
\end{lem}

\begin{rem}
For $0<n \in \nat$,
\begin{enumerate}
\item 
the differential $d$ restricts to
$
S^n (sW) 
\rightarrow 
\bigoplus_{1 \leq k \leq n}  S^k (sW)$;
\item 
the component $k=n$ is the morphism $ S^n (sW) \rightarrow  S^n (sW)$ induced by the differential of the  complex $(sW, d)$; 
\item 
the component $k=1$ identifies with $\dbar : S^n (sW) \rightarrow sW$.
\end{enumerate}
\end{rem}

\begin{nota}
\label{nota:d_nk}
For an $L_\infty$-structure $(S^* (sW),d)$, write $d_{n,k}$ for the component $S^n (sW) 
\rightarrow 
 S^k (sW)$ of the differential. 
\end{nota}

\begin{rem}
\label{rem:ch_cx}
The component $d_{1,1}$ of $d$ yields a chain complex $(sW, d_{1,1})$, which may be denoted simply $(sW,d)$ when no confusion can result.
\end{rem}

\begin{lem}
\label{lem:d2_nl}
Let $(S^*(sW),d)$ be an $L_\infty$-structure. Then, for all $1\leq l \leq n \in \nat$, the composite 
\[
S^n (sW) 
\stackrel{d}{\rightarrow}
\bigoplus_{1 \leq k \leq n}
 S^k (sW)
\stackrel{\bigoplus d_{k,l} }{\rightarrow}
 S^l (sW)
\]
is zero.
\end{lem}

\subsection{Morphisms and isomorphisms of $L_\infty$-structures} 
This section reviews briefly the elements of the theory of $L_\infty$-isomorphisms that are required here.  
 In accordance with Definition \ref{defn:linf}, we use the following:

\begin{defn}
\label{defn:morphism_L_infty}
(Cf. \cite[Chapter 7]{MR2931635} for example.)
A morphism of $L_\infty$-structures from $(S^* (sW), d^W)$ to $(S^* (sU),d^U)$
is a morphism of differential coalgebras  $
  F : (S^* (sW), d^W) \rightarrow (S^* (sU),d^U)$.
 \end{defn}
 
\begin{rem} 
Giving  a morphism $F : S^* (sW) \rightarrow S^* (sU)$ of coalgebras is equivalent to specifying the  components $
F_{n,1} : S^n (sW) \rightarrow sU$ for $1 \leq n \in \nat$; these are linear maps of degree zero.

The morphism $F$ defines a morphism of $L_\infty$-structures if and only if the components $F_{*,1}$ satisfy the usual relations  that encode the compatibility with the differential (see  \cite[Definition 7.1]{MR2931635}, for example).
\end{rem}

A morphism  $F : (S^* (sW), d^W) \rightarrow (S^* (sU),d^U)$ is an isomorphism if and only if $F_{1,1} : sW \rightarrow sU$ is an isomorphism and this gives an isomorphism between the underlying chain complexes $(sW, d_{1,1}^W) \cong (sU, d_{1,1}^U)$ (see \cite[Proposition 7.5]{MR2931635}, for example).

 We focus on the case $W=U$, restricting to isomorphisms such that $F_{1,1}$ is the identity $\id_{sW}$. Rewriting the differentials $d^W$ and $d^U$ as $d$ and $d'$ respectively, the existence of such an isomorphism $F$ implies that $d_{1,1}=d'_{1,1}$.

An isomorphism of coalgebras allows $L_\infty$-structures to be transported (cf. \cite[Exercise 7.3]{MR2931635}, for example):

\begin{lem}
\label{lem:transfer_L_infty}
For $(S^*(sW), d)$ an $L_\infty$-structure and an isomorphism of coalgebras $F: S^*(sW) \stackrel{\cong}{\rightarrow} S^*(sW)$ such that $F_{1,1} = \id_{sW}$, there exists a unique $L_\infty$-structure $(S^*(sW), d')$  with $(sW,d'_{1,1})= (sW,d_{1,1})$ such that $F$ yields an isomorphism of $L_\infty$-algebras
$F: (S^*(sW), d) \rightarrow (S^*(sW), d')$.
\end{lem}

Isomorphisms $F$ such that $F_{1,1} = \id _{sW}$ and $F_{i,1}=0$ for $1 < i \leq n$, where $2 \leq n \in \nat$, play an important rôle in the theory. The following describes how they affect the structure morphisms $d_{i,1}$ for $1 \leq i \leq n$:

\begin{lem}
\label{lem:basic_isos_L_infty}
Let $F : (S^*(sW),d) \stackrel{\cong}{\rightarrow} (S^*(sW),d')$ be an isomorphism of $L_\infty$-structures such that $F_{1,1} = \id _{sW}$ and $F_{i,1}=0$ for $1 < i < n$, where $2 \leq n \in \nat$. Then $d'_{i,1} = d_{i,1}$ for $1 \leq i < n$ and
\begin{eqnarray*}
\label{eqn:d_d'}
d'_{n,1} + d'_{1,1} \circ F_{n,1} = 
d'_{n,1} + d_{1,1} \circ F_{n,1} = 
d_{n,1} + F_{n,1} \circ d_{n,n}.
\end{eqnarray*}
Here $d_{n,n} : S^n(sW) \rightarrow S^n (sW)$ is determined by $d_{1,1}$ by the coderivation property. 
\end{lem} 

\begin{proof}
By construction, $F$ restricts to the identity on $\bigoplus_{k <n} S^k (sW)$ and the restriction of $F$ to $S^n (sW)$ is:
\[
S^n (sW) \stackrel{(\id, F_{n,1})}{\longrightarrow} 
S^n (sW) \oplus sW \subset S^* (sW). 
\] 
The result then follows from the hypothesis that $F$ is a chain map, the first equality being an immediate consequence of the equality $d_{1,1} = d'_{1,1}$. 
\end{proof}

\subsection{From $3$-stubs to $L_\infty$-structures}

The purpose of this Section is two-fold: to introduce the notion of a $3$-stub of an $L_\infty$-structure (see Definition \ref{defn:3_stub}) and then give a criterion that allows a $3$-stub to be extended to a full $L_\infty$-structure (see Proposition \ref{prop:extend}) using basic obstruction theory. For this we work in the following connected context:

\begin{hyp}
\label{hyp:connected}
Suppose that $W$ is $\nat$-graded, so that $sW$ is concentrated in positive degrees. 
\end{hyp}

\begin{rem}
\ 
\begin{enumerate}
\item 
$S^*(-)$ is an exponential functor (cf. Example \ref{exam:S_exp_Lie}) hence one identifies $S^* (sW)$ in (homological) degree $3$ as  
$
S^3 (sW_0)\oplus (sW_1 \otimes sW_0) \oplus sW_2.
$ 
\item 
$S^n (sW)$ is zero in degrees $< n$ and coincides with $S^n (sW_0)$ in degree $n$. 
\end{enumerate}
\end{rem}

The following terminology is adopted here:

\begin{defn}
\label{defn:3_stub}
A $3$-stub is the data for an $L_\infty$-structure on $sW$ up to degree three. Namely 
\begin{enumerate}
\item 
a sequence  $sW_2 \rightarrow sW_1 \rightarrow sW_0$, with $sW_i$ in degree $i+1$; 
\item 
structure morphisms $S^2 (sW_0) \rightarrow sW_0$, $S^3 (sW_0) \rightarrow sW_1$ and $sW_1 \otimes sW_0 \rightarrow sW_1$ of degree $-1$
\end{enumerate}
such that the condition $d^2=0$ is satisfied in degrees $\leq 3$.
\end{defn}

The structure morphisms of an $L_\infty$-structure 
$(S^*(sW),d)$ up to (homological) degree $3$ are indicated in the following diagram:
\[
\xymatrix{
&
&
&
S^3(sW_0)
\ar[ld]
\ar[ldd]|(.3){\gamma}
&
\\
&
&
S^2(sW_0)
\ar[ld]|{\beta}
&
sW_1 \otimes sW_0
\ar[ld]|{\delta}
\ar[l]|{\ }
&
\\ 
&
sW_0
&
sW_1 
\ar[l]|{\alpha}
&
sW_2 
\ar[l]
&
}
\]
Here:
\begin{enumerate}
\item 
$sW_1 \otimes sW_0 \rightarrow S^2 (sW_0)$ is the composite 
\[
sW_1 \otimes sW_0 \stackrel{\alpha \otimes \id} {\rightarrow} sW_0 \otimes sW_0 \stackrel{\mu}{\rightarrow} S^2 (sW_0);
\]
\item 
$  S^3 (sW_0) \rightarrow S^2 (sW_0)$ is the composite 
\[
S^3 (sW_0) \stackrel{\Delta} {\rightarrow}
S^2 (sW_0) \otimes sW_0
 \stackrel{\beta \otimes \id} {\rightarrow} sW_0 \otimes sW_0 \stackrel{\mu}{\rightarrow} S^2 (sW_0).
\]
\end{enumerate}

\begin{rem}
\label{rem:stub}
\ 
\begin{enumerate}
\item
The condition $d^2=0$ (in degrees $\leq 3$) is equivalent to the following:
\begin{enumerate}
\item 
the composite $sW_2 \rightarrow sW_1 \rightarrow sW_0$ is zero; 
\item 
the sum of the two composite morphisms $sW_1 \otimes sW_0 \rightarrow sW_0$  is zero; 
\item 
the sum of the two composite morphisms $S^3 (sW_0) \rightarrow sW_0$ is zero. 
\end{enumerate}
\item
If $W_2=0$, then the first condition is automatic; in this case the data of a $3$-stub is given by the tuple of morphisms
 $(\alpha : sW_1 \rightarrow sW_0, \ \beta  : S^2 (sW_0) \rightarrow sW_0, \  \gamma : S^3 (sW_0) \rightarrow sW_1, \ \delta : sW_1 \otimes sW_0 \rightarrow sW_1)$ such that the second two conditions are satisfied.
\item 
Given any $3$-stub, replacing $W_2$ by zero gives a $3$-stub.
\end{enumerate} 
\end{rem}

The following is a `truncated version' of a standard result for $L_\infty$-structures:

\begin{lem}
\label{lem:lie}
\ 
\begin{enumerate}
\item 
For $W_0$, $W_1$ $\kring$-vector spaces concentrated in degrees $0$ and $1$ respectively
and the data of a $3$-stub $(\alpha : sW_1 \rightarrow sW_0, \ \beta  : S^2 (sW_0) \rightarrow sW_0, \  \gamma : S^3 (sW_0) \rightarrow sW_1,\ \delta : sW_1 \otimes sW_0 \rightarrow sW_1)$ taking $W_2=0$,  the morphism $\beta : S^2 (sW_0) \rightarrow sW_0$ induces a Lie algebra structure on $H_0 := s^{-1} \coker \big(sW_1 \rightarrow sW_0\big)$. 
\item 
Conversely, given a Lie algebra $H_0$ and an exact sequence of $\kring$-vector spaces $W_1 \rightarrow W_0 \rightarrow H_0 \rightarrow 0$, there exists a $3$-stub structure inducing this with $W_2=0$ and  $\alpha$ the suspension of the given map $ W_1 \rightarrow W_0 $.
\end{enumerate}
\end{lem}

The following Proposition  is a variant, starting from a given $3$-stub, of  transfer of structure results (cf. the homotopy transfer theorem \cite[Section 10.3]{LV}). It is proved by basic obstruction theory.

\begin{prop}
\label{prop:extend}
Let $W$ be a $\nat$-graded $\kring$-vector space equipped with the following structure:
\begin{enumerate}
\item 
a chain complex  $(sW,d)$ such that $H_*(sW,d)=0$ for $* \neq 1$; 
\item 
morphisms  $ \beta : S^2 (sW_0) \rightarrow sW_0$, $\gamma : S^3 (sW_0) \rightarrow sW_1$, $ \delta : sW_1 \otimes sW_0 \rightarrow sW_1$ defining a $3$-stub.
\end{enumerate}
Then this structure extends to an $L_\infty$-structure $(S^* (sW),d)$. 
\end{prop}

\begin{proof}
The construction of the $L_\infty$-structure is by increasing induction upon the (polynomial) degree $n \geq 1$ in the symmetric algebra and on the homological degree. The initial cases are given by  $n=1$, which  is given as the chain complex $(sW, d)$ and by homological degree $\leq 3$, which is given  by the $3$-stub.

For the inductive step, we may assume that $n>1$ and that the homological degree is $>3$ and that 
 the result has been proven for $n'<n$ in all homological degrees and for the given $n$ in smaller homological degrees.
 
We require to construct $S^n (sW) \rightarrow sW$ in the given homological degree (here the homological degree is not indicated in the notation), 
 so that the condition of Lemma \ref{lem:diff_coderiv} is satisfied, i.e., 
 \[
S^n (sW) 
\stackrel{d}{\rightarrow}
\bigoplus_{1 \leq k \leq n}
 S^k (sW)
\stackrel{\bigoplus  d_{k,1} }{\rightarrow}
sW
\]
is zero.

Consider the morphism 
\begin{eqnarray}
\label{eqn:comp_k>1}
S^n (sW) 
\stackrel{d}{\rightarrow}
\bigoplus_{1 < k \leq n}
S^k (sW)
\stackrel{\bigoplus  d_{k,1} }{\rightarrow}
sW
\end{eqnarray}
given by omitting the (unknown) component $k=1$. 

We claim that the inductive hypotheses imply that the composite of (\ref{eqn:comp_k>1}) with the differential $d : sW \rightarrow sW$ is zero. To see this, for each $k$ ($1< k \leq n$) the inductive hypothesis (using the homological degree in the case $k=n$) ensures  that the composite $S^k (sW) \stackrel{\bigoplus  d_{k,1} }{\rightarrow} sW \stackrel{d}{\rightarrow} sW$ is equal to the negative of 
\[
S^k (sW) 
\stackrel{\bigoplus d_{k,l}}{\rightarrow}
\bigoplus_{1 < l \leq k}
S^l (sW)
\stackrel{\bigoplus s d_{l,1} }{\rightarrow}
sW.
\]

In addition, the inductive hypothesis implies (as for Lemma \ref{lem:d2_nl}) that, for each $l>1$, the composite
\[
S^n (sW) 
\stackrel{d}{\rightarrow}
\bigoplus_{1 < k \leq n}
S^k (sW) 
\stackrel{\bigoplus d_{k,l}}{\rightarrow}
S^l (sW)
\]
is zero. Putting these facts together establishes the claim. 

Since the composite (\ref{eqn:comp_k>1}) maps to cycles of $(sW,d)$ and the homology of this complex is concentrated in degree $1$, the composite maps to boundaries (using the hypothesis that the homological degree in $S^n (sW)$ is greater than $3$). Thus there exists a linear map 
\[
S^n (sW) \rightarrow sW
\]
such that the composite with $d: sW \rightarrow sW$ gives the negative of (\ref{eqn:comp_k>1}). This completes the construction of the required map and hence the proof of the inductive step.
\end{proof}

The $L_\infty$-structure constructed in Proposition \ref{prop:extend} is unique up to $L_\infty$-isomorphism. This is a consequence of the following result, again proved by basic obstruction theory:

\begin{prop}
\label{prop:unique}
Let $W$ be a $\nat$-graded $\kring$-vector space equipped with $L_\infty$-structures $(S^* (sW),d)$ and $(S^*(sW),d')$ such that the following hold:
\begin{enumerate}
\item 
$d_{1,1} = d'_{1,1}$;
\item 
the restrictions of $d_{2,1}$ and $d'_{2,1}$ to $S^2(sW_0) \subset S^2 (sW)$  coincide; 
\item 
 $H_*(sW,d_{1,1})=0$ for $* \neq 1$. 
\end{enumerate}
Then $(S^* (sW),d)$ and $(S^*(sW),d')$ are $L_\infty$-isomorphic by an isomorphism which restricts to the identity on $sW$ and on $S^2(sW_0)$.
\end{prop}

\begin{proof}
If $d = d'$, then the result is immediate. Otherwise, choose $n$ minimal such that $d_{n,1} \neq d'_{n,1}$; the hypothesis on $d_{1,1}$ ensures that $n \geq 2$. Consider the lowest homological degree $t$ in which $d_{n,1} \neq d'_ {n,1}$; $t \geq 3$, since $S^n (sW)$ is concentrated in homological degree $\geq n$ and, in the case $n=2$, the hypothesis excludes homological degree $2$. 

On homological degree $t$, a standard argument (cf. the proof of Proposition \ref{prop:extend})   shows that the difference 
\[
(d_{n,1} - d'_{n,1}) \ : \  S^n (sW) \rightarrow sW 
\]
maps to $d_{1,1}$-cycles in homological degree $t-1 \geq 2$. 

By the hypothesis on $H_*(sW,d_{1,1})$, the differential $d_{1,1}$ surjects onto the cycles of degree $t-1$; this surjection admits a linear section. Composing with this section, one has a linear map $f : S^n (sW) \rightarrow sW$ of degree zero, non-zero only in homological degree $t$, such that 
\[
d_{n,1} - d'_{n,1} = d_{1,1} \circ f
\]
 on homological degree $t$. 

Let $F$ be the automorphism of the coalgebra $S^* (sW)$ with components $F_{1,1}= \id _{sW}$, $F_{i,1}=0$ for $1< i <n$ and $F_{n,1}=f$. Using  transport of $L_\infty$-structures with respect to $F$ (see Lemma \ref{lem:transfer_L_infty}) and Lemma \ref{lem:basic_isos_L_infty},  one shows that $(S^*(sW),d')$ is isomorphic to an $L_\infty$-structure $(S^* (sW),d'')$ such that $d''_{i,1}= d_{i,1}$ for $i<n$ and, for $i=n$, up to and including homological degree $t$. 
 This isomorphism satisfies the following stability property: it acts as the identity on $\bigoplus _{i <n} S^i (sW)$ and on $S^n (sW)$ in homological degrees $<t$. 

The stability property implies that this construction can be iterated to give the required isomorphism between $(S^* (sW),d)$ and $(S^*(sW),d')$ (cf. the argument used in the proof of \cite[Theorem 7.9]{MR2931635}, for example). Namely, for $n$ as above, one can iterate the isomorphisms obtained as above, with respect to increasing homological degree, to show that $(S^*(sW),d')$ is isomorphic to an $L_\infty$-structure $(S^* (sW),d'')$ such that $d''_{i,1}= d_{i,1}$ for $i\leq n$. Likewise, this argument can be iterated with respect to increasing $n$.
\end{proof}

\begin{rem}
\label{rem:linfty_iso_not_unique}
Proposition \ref{prop:unique} does not assert that there is a  unique such $L_\infty$-isomorphism between $(S^* (sW),d)$ and $(S^*(sW),d')$. Indeed, the map $f$ used in the proof
 is only unique up to a map taking values in $\ker d_{1,1}$. 
\end{rem}

\begin{rem}
When applied to the situation of Proposition \ref{prop:extend}, if $(S^* (sW), d)$ and $(S^*(sW),d')$ have the same underlying $3$-stub, then the isomorphism constructed in Proposition \ref{prop:unique} leaves this invariant. However, Proposition \ref{prop:unique} is more general: it allows the $3$-stub to be varied  with only  $d_{2,1}= d'_{2,1}$ in homological degree two fixed.

One could generalize still further by allowing automorphisms of the coalgebra $S^* (sW)$ acting as the identity on $sW$ but that do not respect $d_{2,1}$ in homological degree two. This is not pursued here, since there is a canonical choice for $d_{2,1}$ in the case of interest.
\end{rem}

\section{The natural $L_\infty$-structure on $(\lieopd (s\g) , \dleib)$ }
\label{sect:main}

The purpose of this Section is to show that there exists a natural $L_\infty$-structure on $(\lieopd (s\g) , \dleib)$ (see Theorem \ref{thm:main}). This involves first constructing an explicit $3$-stub and then extending this to the required natural $L_\infty$-structure.

\subsection{The $3$-stub}

The construction of the $3$-stub below (see Proposition \ref{prop:leib_stub})  uses the following relations and notation
 for  a  Leibniz algebra $\g$. The Leibniz relation can be written in the equivalent forms:
\begin{eqnarray*}
(xy) z &=& (xz) y + x (yz) \\
x (yz) &=& (xy)z - (xz)y. 
\end{eqnarray*}
The second leads to  the standard:

\begin{lem}
\label{lem:right_action}
For $x,y ,z  \in \g$, $x(yz + zy ) =0$.
\end{lem}

\begin{nota}
Write $\cop$ for the sum under cyclic $\zed/3$ permutation so that,  for example, 
$ x(yz) \cop : = x(yz) + y (zx) + z (xy). $ 
\end{nota}

\begin{defn}
\label{defn:lp}
Let $\lp : \g \otimes \g \rightarrow \g$ be the operation defined for $x, y \in \g$  by
$
x \lp y := xy -yx.
$
\end{defn}

\begin{lem}
\label{lem:circ_identities}
For $x, y , z \in \g$, there are equalities:
\begin{eqnarray*}
 x(yz) \cop
&= & (x \lp y )z \cop \\
x (y \lp z) \cop &=& 2 (x \lp y) z \cop \\
\big ( x (y \lp z) + (y \lp z) x \big) \cop &=& 3 (x \lp y) z \cop \\
(x \lp y ) \lp z \cop &=& -  (x\lp y) z \cop .
\end{eqnarray*}
\end{lem}

\begin{proof}
Direct calculation using the Leibniz relation gives the first equality. 
The second follows by combining this with  Lemma \ref{lem:right_action}  
and the third then follows by taking into account the cyclic invariance. 

The second equality also implies 
$$
\big((x \lp y ) z - z (x \lp y)\big) \cop = - (x\lp y) z \cop,
$$
which gives the final equality.
\end{proof}

Consider  Pirashvili's subcomplex $(\lieopd (s\g) , \dleib)
\subset 
(T(s \g), \dleib) $ of the Leibniz complex. The underlying graded $\kring$-vector space $\lieopd (s\g)$ is concentrated in positive degrees.

Using the notation of Section \ref{sect:lift}, set $W:= s^{-1} \lieopd (s\g)$ (so that $sW  = \lieopd (s\g)$) and take $(sW,d)$ to be  $(\lieopd (s\g) , \dleib)$.

\begin{prop}
\label{prop:leib_stub}
For $\g$ a Leibniz algebra, the truncation $\lieopd_3 (s\g) \rightarrow \lieopd_2 (s\g) \rightarrow s\g$ of 
 Pirashvili's complex has a natural $3$-stub structure
\[
\xymatrix{
&
&
&
S^3(s\g)
\ar[ld]
\ar[ldd]
&
\\
&
&
S^2(s\g)
\ar[ld]
&
\lieopd_2(s\g) \otimes s\g 
\ar[ld]
\ar[l]|{\ }
&
\\ 
&
s\g
&
\lieopd_2(s\g) 
\ar[l]^{\dleib}
&
\lieopd_3(s\g)
\ar[l]^{\dleib}
&
}
\]
with defining structure maps:
\begin{enumerate}
\item 
$S^2 (s\g)\rightarrow s \g$ the map $(sx) (sy) \mapsto - s(x\lp y)$ for $x, y \in \g$; 
\item 
$\dlie_{2,1} : \lieopd_2(s\g) \otimes s\g \rightarrow \lieopd_2 (s\g) $; 
\item 
$S^3 (s\g)\rightarrow \lieopd_2 (s \g)$  the map $(sx)(sy)(sz) \mapsto - \frac{1}{3} [ s(x\lp y) ,sz] \cop$ for $x,y,z \in \g$.
\end{enumerate}

The induced Lie algebra structure on the homology is the canonical Lie algebra structure on $\glie$ (up to sign).
\end{prop}

\begin{proof}
It is clear that the structure morphisms are natural with respect to the Leibniz algebra $\g$. The relations are checked by direct calculation as follows. 

Consider an element $[sx , sy ] \otimes sz \in \lieopd_2(s\g) \otimes s\g $. Applying $\dlie_{2,1}$ gives 
$[s(xz),sy] + [sx, s(yz)] \in \lieopd_2(s\g)$ which has image $- s \big(  (xz)y + y (xz) + x(yz) + (yz) x  \big)$ under the Leibniz differential. Using the Leibniz relations $y(xz) = (yx) z - (yz) x$ and $x( yz) = (xy)z - (xz)y$, this is equal to 
$ -s \big((xy + yx)z \big)$. 

 The differential  $ \lieopd_2(s\g) \otimes s\g \rightarrow S^2 (s \g)$ is induced by the Leibniz differential $\lieopd_2 (s\g) \rightarrow s\g$, and the image of $  [sx , sy ] \otimes sz$ is $- (s(xy+yx))(sz) \in S^2 (s\g)$. Applying the differential $S^2 (s\g)\rightarrow s\g$ gives $s \big( (xy+yx) z - z (xy+yx) \big) = s \big((xy+ yx)z\big)$, since right multiplication by $(xy+yx)$ is zero, by Lemma \ref{lem:right_action}. It follows that the sum of the two maps $ \lieopd_2(s\g) \otimes s\g \rightarrow s \g$ is zero, as required. 
 
Now consider an element $(sx)(sy)(sz)$ of $S^3(s\g)$. The image of this element under $S^3(s\g)\rightarrow S^2(s\g)$ is 
$- (s(x\lp y))(sz) \cop \in S^2 (s\g)$. Under $S^2 (s\g) \rightarrow s\g$, this has image $s \big( (x\lp y ) \lp z \big)\cop$. By Lemma \ref{lem:circ_identities}, this is equal to $- s\big( (x\lp y)z\big)\cop$.

The image of  $(sx)(sy)(sz)$ under $S^3(s\g) \rightarrow \lieopd_2 (s\g)$ is $-\frac{1}{3} [s (x\lp y), sz ] \cop$. Applying the Leibniz differential $\lieopd_2 (s\g) \rightarrow s\g$ then gives $\frac{1}{3} s \big( (x\lp y)z + z (x\lp y) \big) \cop$, which is equal to $s \big( (x \lp y)z\big)\cop$ by Lemma  \ref{lem:circ_identities}. Thus, the sum of the two maps $S^3 (s\g) \rightarrow s\g$ is zero, as required.
 \end{proof}
 
\begin{rem}
The choice of sign for the map $S^2 (s \g) \rightarrow s\g$ was dictated by the differential of the Leibniz complex. Namely, this arises as the restriction of $\dleib : s\g \otimes s\g \rightarrow s\g$ to the symmetric invariants, since $\dleib (sx \otimes sy) = -s (xy)$.
\end{rem}

\begin{cor}
\label{cor:l_infty_free}
For $\g = \leibopd (V)$ a free Leibniz algebra, there is an $L_\infty$-structure on $(s^{-1} \lieopd (s\leibopd (V)), d)$ such that the restriction of $d_{2,1}$ to $S^2 (s \leibopd (V))$ is the map $S^2 (s \leibopd (V)) \rightarrow s \leibopd (V)$ given by $(sx)(sy) \mapsto -s (x \star y)$. 
 This structure is unique up to $L_\infty$-isomorphism.
\end{cor}

\begin{proof}
Existence follows from Proposition \ref{prop:extend}, using Proposition \ref{prop:leib_stub} to provide the $3$-stub and Theorem \ref{thm:liezation} for the homological input. The unicity follows from Proposition \ref{prop:unique}.
\end{proof}

This is refined below to obtain naturality with respect to $V$ and hence the extension to all Leibniz algebras.

\subsection{The natural $L_\infty$-structure}

The proof of the following Theorem relies heavily on the  naturality with respect to the Leibniz algebra $\g$, which allows reduction to the case of free Leibniz algebras. Appendix \ref{sect:naturality} provides the key ingredient underlying  the argument, based upon viewing operads as natural operations on the associated category of algebras.

\begin{thm}
\label{thm:main}
For $\g$ a Leibniz algebra, there is a natural $L_\infty$-structure on the desuspension of the Pirashvili complex $s^{-1}(\lieopd (s\g) , \dleib)$ such that the restriction of $d_{2,1}$ to $S^2 (s \g)\subset  S^2 (\lieopd (s\g))$ is the map $S^2 (s \g) \rightarrow s \g$ given by $(sx)(sy) \mapsto -s (x \star y)$. In particular, this induces the canonical Lie algebra structure on $\glie$ in homology (up to sign).

This natural structure is unique up to natural $L_\infty$-isomorphism.
\end{thm}

Recall that $\vectf \subset \vect$ denotes the full subcategory of finite-dimensional vector spaces. As in Appendix \ref{sect:naturality}, $\leibalgf \subset \leibalg$ denotes the full subcategory of free Leibniz algebras $\leibopd (V)$ where $V \in \ob \vectf$, so that $\leibalgf$ is essentially small.

\begin{nota}
Natural transformations between functors from $\vectf$ to $\vect$ are denoted $\nt_{V \in \vectf}$ and, for functors from $\leibalg$ to $\vect$ restricted to $\leibalgf$, are denoted by $\nt_{\g \in \leibalgf}$.
\end{nota}

First observe that the extension as coderivations given in Remark \ref{rem:coderiv_extn} passes to natural transformations:

\begin{lem}
\label{lem:coderiv_extn_nt}
The extension of linear maps as coderivations induces a $\kring$-linear map
\[
\nt_{\g \in \leibalgf} (S^* (\lieopd (s\g)), \lieopd (s\g)) 
\rightarrow 
\nt_{\g \in \leibalgf} (S^* (\lieopd (s\g)), S^* (\lieopd (s\g))) 
\]
that admits a natural retract induced by the natural projection $S^* (\lieopd (s\g) ) \twoheadrightarrow \lieopd (s\g)$.
\end{lem}

\begin{rem}
\label{rem:coderiv_extn_nt}
The result holds more generally: the restriction to $\leibalgf$ is only imposed so that the result can be stated using morphisms in the  category of functors from $\leibalgf$ to $\vect$.
\end{rem}

Restriction associated to the unit $V \hookrightarrow \leibopd (V)$ gives the following:

\begin{prop}
\label{prop:compatibility_extn_restriction}
Restriction induces the $\kring$-linear isomorphism
\[
\xymatrix{
\nt_{\g \in \leibalgf} (S^* (\lieopd (s\g)), \lieopd (s\g)) 
\ar[r]^(.45)\cong
&
\nt_{V \in \vectf} (S^* (\lieopd (sV)), \lieopd (s\leibopd (V))).
}
\]
\end{prop}

\begin{proof}
The  morphisms are induced respectively by restriction (with respect to the first variable)  along the free Leibniz algebra functor $\leibopd : \vectf \rightarrow \leibalgf$ and by restriction along the unit map $V \hookrightarrow \leibopd (V)$. That the composite is an isomorphism follows from Proposition \ref{prop:refined_prop_operations}.
\end{proof}

\begin{cor}
\label{cor:nat_acyclic}
The complex $\big( \nt_{\g \in \leibalgf} (S^* (\lieopd (s\g)), \lieopd_\bullet (s\g)), d \big)$
with differential 
$$
d :   \nt_{\g \in \leibalgf} (S^* (\lieopd (s\g)), \lieopd_t (s\g))
\rightarrow 
\nt_{\g \in \leibalgf} (S^* (\lieopd (s\g)), \lieopd_{t-1} (s\g))
$$
 induced by the Leibniz differential 
$\dleib :   \lieopd_t (s\g) \rightarrow \lieopd_{t-1} (s\g)$ has homology concentrated in degree one, where it is isomorphic to 
\[
\nt_{V \in \vectf} (S^* (\lieopd (sV)), s\lieopd (V)).
\]
\end{cor}

\begin{proof}
By  Proposition \ref{prop:compatibility_extn_restriction}, the given chain complex is isomorphic to 
\[
\big( \nt_{V \in \vectf} (S^* (\lieopd (sV)), \lieopd_\bullet (s\leibopd (V))), d \big)
\]
with differential given by $\dleib : \lieopd_t (s\leibopd (V)))
\rightarrow \lieopd_{t-1} (s\leibopd (V)))$. 

By Theorem \ref{thm:liezation}, Pirashvili's complex gives the acyclic complex 
\[
\ldots \rightarrow \lieopd_2 (s\leibopd (V)))
\rightarrow \lieopd_{2} (s\leibopd (V))
\rightarrow 
s\leibopd (V)
\rightarrow 
\lieopd (V)
\rightarrow 
0,
\]
where the term $\lieopd (V)$ in homological degree zero corresponds to the degree one homology of the Pirashivili complex.

This is an exact sequence of analytic functors on $\vect$ (see Appendix \ref{sect:poly} for a brief review of analytic functors). This sequence splits naturally as  functors on $\vect$ by the semisimplicity of the category of homogeneous polynomial functors over a field of characteristic zero (see Corollary \ref{cor:semisimple}). Explicitly, for each $t\geq 1$, the short exact sequence 
\begin{eqnarray}
\label{eqn:ses_pira_complex}
0
\rightarrow 
\ker \dleib
\rightarrow 
 \lieopd_t (s\leibopd (V)))
 \rightarrow 
 \mathrm{image} \ \dleib 
 \rightarrow 
 0
\end{eqnarray}
splits as functors of $V$ and these short exact sequence splice together to give the exact complex.

Upon applying $\nt_{V \in \vectf} (  S^* (\lieopd (sV)), - ) $, one obtains the result.
\end{proof}

\begin{proof}[Proof of Theorem \ref{thm:main}]
Existence is proved by refining the proof of Proposition \ref{prop:extend}, starting from the initial structure provided by Proposition \ref{prop:leib_stub}. 

First one considers the natural $L_\infty$-structure for the category $\leibalgf$. The structure morphisms in 
 \[
 \nt_{\g \in \leibalgf} (S^n (\lieopd (s\g)), \lieopd (s\g)) 
 \]
are constructed by increasing induction on $n\geq 1$ and on the homological degree. The case $n=1$ is given by the Pirashvili complex and the behaviour in homological degree $\leq 3$ is given by the $3$-stub of  Proposition \ref{prop:leib_stub}.

The inductive step is proved using basic obstruction theory as in the proof of  Proposition \ref{prop:extend}, but working with natural transformations. The vanishing of the relevant obstruction groups is given by Corollary \ref{cor:nat_acyclic}.

The passage to the general case uses Corollary \ref{cor:large_nat}, which provides the natural structure morphisms in 
\[
\hom_\kring  (S^n (\lieopd (s\g)), \lieopd (s\g)) 
\]
for an arbitrary Leibniz algebra $\g$. It remains to show that these induce a natural $L_\infty$-structure on $\g$.
 First one checks that the extension of these structure morphisms to a coderivation $d$ of $S^*(s \g)$ is compatible with the natural $L_\infty$-structure for $\g \in \ob \leibalgf$; this follows from the generalization of Lemma \ref{lem:coderiv_extn_nt} indicated in Remark \ref{rem:coderiv_extn_nt}. 

Finally one checks that $d^2=0$; this is proved using the compatibility of the map of Corollary \ref{cor:large_nat} with composition, as explained in Remark \ref{rem:naturality_large_nat}.

The proof of unicity is a refinement of that of Proposition \ref{prop:unique}, first working with natural transformations on $\leibalgf$. Corollary \ref{cor:nat_acyclic} ensures that the map $f$ appearing in the proof of Proposition \ref{prop:unique} can be constructed as a natural transformation. Moreover, one has the following counterpart of Lemma \ref{lem:coderiv_extn_nt}: the associated $F$ is a natural automorphism of coalgebras. The unicity of the transported $L_\infty$-structure given by Lemma \ref{lem:transfer_L_infty} thus concludes the refinement of the key step in the proof of Proposition \ref{prop:unique}. This establishes the result on the full subcategory $\leibalgf$.

The passage to general Leibniz algebras proceeds as above, {\em mutatis mutandis}, by using Corollary \ref{cor:large_nat}.
 \end{proof}
  
\begin{rem}
\label{rem:non_unicity_natural}
Working with natural transformations for functors defined on $\leibalgf$ rigidifies the situation somewhat; however, the $L_\infty$-isomorphism appearing in the unicity statement is still not unique (cf. Remark \ref{rem:linfty_iso_not_unique}).

The non-unicity stems from the fact that the natural splitting of the short exact sequence (\ref{eqn:ses_pira_complex}) used in the proof of Corollary \ref{cor:nat_acyclic} is not in general unique. To see that this does indeed lead to non-unicity, one has to check that $ \nt_{\g \in \leibalgf} (S^* (\lieopd (s\g)), -)$ detects this.
\end{rem}

 \subsection{An explicit choice for the column $n=2$}

 The initial step in constructing a natural $L_\infty$-structure as in Theorem \ref{thm:main}  
  is the construction of the natural structure morphisms $S^2 (\lieopd (s \g)) \rightarrow s \lieopd (s \g)$. 
An explicit choice for these is exhibited here.
 
Using the exponential property of $S^* (-)$, one has: 
 
 \begin{lem}
 There is a natural isomorphism:
 \[
 S^2 (\lieopd (s \g))
 \cong 
 \Big( \bigoplus_{n\geq 1} S^2 (\lieopd_n (s\g) )  \Big) 
 \oplus 
 \Big ( \bigoplus_{m> n \geq 1} \lieopd_m (s\g) \otimes \lieopd_n (s\g) \Big),
 \]
where  $S^2 (\lieopd_n (s\g) )$ is in homological degree $2n$ and  $\lieopd_m (s\g) \otimes \lieopd_n (s\g)$ is in homological degree $m+n$.
 \end{lem}

 The Leibniz differential on $\lieopd (s\g)$ makes $S^2 (\lieopd (s \g))$ into a complex. 
 The differential acts via the induced morphisms:
 \begin{enumerate}
 \item 
 $S^2 (\lieopd_n (s\g )) \rightarrow  \lieopd_{n} (s\g) \otimes \lieopd_{n-1} (s\g)$ for $n \geq 1$; 
 \item 
 \[
 \xymatrix{
 \lieopd_m (s\g) \otimes \lieopd_n (s\g)
 \ar[r]
 \ar[d]
 &
 \lieopd_m (s\g) \otimes \lieopd_{n-1} (s\g)
 \\
 \lieopd_{m-1} (s\g) \otimes \lieopd_n (s\g);
 }
 \]
 if $m=n+1$, the vertical map is further composed with $\lieopd_{n} (s\g) \otimes \lieopd_n (s\g) \stackrel{\mu}{\rightarrow} S^2 (\lieopd_n (s\g))$.
 \end{enumerate}

 \begin{prop}
 There exists a natural $L_\infty$-structure as in Theorem \ref{thm:main} extending the $3$-stub of Proposition \ref{prop:leib_stub} such that the components $S^2 (\lieopd (s\g)) \rightarrow \lieopd (s\g)$ of the $L_\infty$-structure are:
 \begin{enumerate}
 \item 
 the morphism $S^2 (s\g) \rightarrow s (s\g)$ given by the $3$-stub;
 \item 
 the morphism $\dlie : \lieopd_m (s\g) \otimes s\g \rightarrow s \lieopd_m (s\g)$ for $m>1$ (taking $n=1$ so that $s\g = \lieopd_1 (s\g)$);
 \item 
 zero otherwise.
 \end{enumerate}
 \end{prop}
 
 \begin{proof}
 In homological degrees $\leq 3$, the $L_\infty$-structure morphisms are imposed by Proposition \ref{prop:leib_stub}.

The result is then proved by increasing induction on homological degree. There are two ingredients:
\begin{enumerate}
\item 
the compatibility of $\dlie$ with the Leibniz differential (see Remark \ref{rem:dleib_glie-linear});
\item 
the vanishing of $\dlie$ on restriction to $\lieopd_m (s\g) \otimes \big( \mathrm{image} (\lieopd_2 (s\g) \rightarrow s\g)\big)$.
\end{enumerate} 
 \end{proof}

\appendix
\section{Polynomial and analytic functors}
\label{sect:poly}

In this Section polynomial and analytic functors from $\vect$ to $\vect$ over a field $\kring$ of characteristic zero are reviewed in sufficient generality for the applications here. Standard references are \cite[Appendix A to Chapter I]{MR3443860} for polynomial functors  and \cite[Chapitre 4]{MR927763} for polynomial and analytic functors.
 
\subsection{Polynomial functors}

\begin{defn}
\label{defn:poly}
For $d \in \nat$ and $M$ a right $\sym_d$-module, the associated Schur functor $\tilde{M} : \vect \rightarrow \vect$ is given by
\[
V \mapsto 
\tilde{M}(V) := M\otimes_{\sym_d} V^{\otimes d}.
\]
This is a (homogeneous) polynomial functor of degree $d$.  A polynomial functor is a functor that is a finite direct sum of homogeneous polynomial functors (possibly of different degree).
\end{defn}

\begin{exam}
The functor $V \mapsto \lieopd_d (V)$ is homogeneous polynomial of degree $d$. 
\end{exam}

\begin{prop}
\label{prop:tensor_comp_poly}
Let $F, G : \vect \rightarrow \vect$ be homogeneous polynomial  functors of degree $d$ and $e$ respectively. Then 
\begin{enumerate}
\item 
the tensor product $F \otimes G$ defined by $(F\otimes G)(V):= F(V) \otimes G(V)$ is homogeneous polynomial of degree $d+e$; 
\item 
the composite $F \circ G$ defined by $F\circ G (V):= F(G(V))$ is homogeneous polynomial of degree $de$.
\end{enumerate}
\end{prop}

\begin{rem}
Polynomial functors are determined by their restriction to the full subcategory $\vectf \subset 
\vect$ of finite-dimensional vector spaces. (More precisely, a polynomial functor is equivalent to the left Kan extension of its restriction to $\vectf$.) The category $\vectf$ is essentially small, so this allows consideration of the category of functors from $\vectf$ to $\vect$; morphisms of this category are denoted $\nt_{V\in \vectf}$.  
\end{rem}

\begin{prop} 
\label{prop:schur_fully_faithful}
For $d \in \nat$, the Schur functor from $\sym_d\op$-modules to endofunctors of $\vect$ 
 is fully faithful. In particular,  for $M$ a right $\sym_d$-module, there is a natural isomorphism of $\sym_d\op$-modules:
\[
M \cong 
 \nt _{V\in \vectf} (V^{\otimes d}, \tilde{M}(V))
\] 
where the $\sym_d\op$-action on the right hand side is induced by place permutations of the tensor factors $V^{\otimes d}$.
\end{prop}

Since $\kring [\sym_d]$ is semisimple for $\kring$ a field of characteristic zero, one deduces:

\begin{cor}
\label{cor:semisimple}
For $d \in \nat$, the category of homogeneous polynomial functors of degree $d$ is semisimple.
\end{cor}

\begin{rem}
There is no interaction between homogeneous polynomial functors of different degree; this is a consequence of the fact that $\nt_{V \in \vect} (V^{\otimes d}, V^{\otimes e})=0$ if $d\neq e$. It follows that the category of polynomial functors is semisimple.
\end{rem}

\subsection{Analytic functors}
The Schur functor construction extends to treating sequences $\{M(n) | n \in \nat \}$  where $M(n)$ is a $\sym_n\op$-module. 
Namely, the associated Schur functor
\[
\tilde{M} (V) 
:= 
\bigoplus_{n\in \nat}
M(n) \otimes_{\sym_n} V^{\otimes n}
\]
is the direct sum of the Schur functors for the  modules $M(n)$. 

By definition, an analytic functor is one of this form. In particular, a polynomial functor is analytic.

\begin{exam}
\label{exam:leib_opd_schur}
Leibniz algebras are governed by the Leibniz operad $\leibopd$. In particular, the associated functor $V \mapsto \leibopd (V)$ is analytic.  The free Leibniz algebra on $V$ was described by Loday and Pirashvili \cite[Lemma 1.3]{LP}:
 \[
 \leibopd (V) \cong \bigoplus_{n\geq 1} T^n (V). 
 \] 
 
 In particular, this identifies the terms of the Leibniz operad as $\sym_n\op$-modules 
 as $\leibopd(0)=0$ and $\leibopd(n) = \kring [\sym_n]$ for $n \geq 1$.
 \end{exam}

\begin{nota}
\label{nota:poly}
For an analytic functor $F$ on $\vect$, write $F_{[d]}$ for the direct summand of polynomial degree $d \in \nat$,  so that $F\cong \bigoplus_{d \in \nat} F_{[d]}$.
\end{nota}

\begin{rem}
The construction $(-)_{[d]}$ gives an exact functor from analytic functors to homogeneous polynomial functors of degree $d$.
\end{rem}

Proposition \ref{prop:tensor_comp_poly} has the following counterpart for analytic functors:

\begin{prop}
Let $F, G : \vect \rightarrow \vect$ be analytic functors. Then 
$F \otimes G$ and $F \circ G$
are analytic.
\end{prop}

\begin{exam}
\ 
\begin{enumerate}
\item 
The functor $V \mapsto \lieopd (V)$ is analytic, with $\lieopd (V)_{[d]} = \lieopd_d (V)$.
\item 
The functor $V\mapsto S^* (V)$ is analytic, with $S^* (V)_{[d]} = S^d (V)$.
\end{enumerate}
Hence the composite $V \mapsto S^* (\lieopd (V))$ is analytic. 
\end{exam}

\begin{exam}
\label{exam:S_exp_Lie}
The functor $S^*$ is exponential, i.e., there is a natural isomorphism $S^* (V \oplus W) \cong S^* (V) \otimes S^* (W)$. For example, this gives the natural isomorphism $ 
S^* (\lieopd (V))
\cong 
\bigotimes_{n\geq 1} 
S^* (\lieopd_n (V))$. This can be developed to calculate the homogeneous components of $S^* (\lieopd (V))$.
\end{exam}

\section{Naturality}
\label{sect:naturality}

This Section recalls how to understand natural transformations between polynomial functors evaluated on algebras over an operad; in particular, no claim to originality is made. The main result, Proposition \ref{prop:refined_prop_operations}, is used in the proof of Theorem \ref{thm:main}.

\subsection{The operadic case}

First recall the operadic framework for studying operations on algebras (see \cite{LV} for definitions). Let $\opd$ be an operad in $\kring$-vector spaces with unit $I \rightarrow \opd$. The category of $\opd$-algebras is denoted $\opdalg$.

\begin{rem}
\label{rem:small}
One would like to work with the `category of functors' from $\opdalg$ to $\vect$. However,   the category $\opdalg$ is not essentially small. To avoid introducing universes, one restricts to the full subcategory $\opdalgf$ of $\opdalg$ with objects free $\opd$-algebras of the form $\opd (V)$ with $V \in \ob \vectf$; by construction,  this subcategory has a small skeleton. The morphisms in the category of functors from $\opdalgf$ to $\vect$ are denoted by $\nt_{A \in \opdalgf}$.

Since the only functors  considered here are given by composing the forgetful functor 
 $\opdalg \rightarrow \vect$ with a polynomial functor (or a graded variant of such), Corollary \ref{cor:large_nat} shows that working with the restriction to $\opdalgf$ is sufficient for current purposes. 
\end{rem}

An $\opd$-algebra structure on a $\kring$-vector space $V$ is given by a morphism of operads $
\opd 
\rightarrow 
\opend (V)$,   where $\opend (V)$ is the endomorphism operad for $V$. 
In particular, one can view $\opd (n)$ as the $\kring$-vector space of $n$-ary operations on 
 $\opdalg$ by exploiting the  functors   $A \mapsto A^{\otimes n}$ from $\opdalg$ to $\vect$, for $n \in \nat$.
  The $\opd$-algebra structure induces 
a  morphism of $\sym_n\op$-modules
\[
\opd (n) 
\rightarrow 
\hom_{\vect} (A^{\otimes n}, A),
\]
where the $\sym_n\op$ action on the codomain is induced by the place permutation action on $A^{\otimes n}$. 
 Restricting to $\opdalgf$, this gives the morphism of $\sym_n\op$-modules
 \begin{eqnarray*}
 \opd (n) 
 \rightarrow 
 \nt _{A \in \opdalgf} (A^{\otimes n}, A).
 \end{eqnarray*}

 This is seen to be a natural isomorphism as follows. For $V \in \ob \vectf$, $\opd (V)$ is the free $\opd$-algebra on $V$ given by the Schur functor $\opd (V) = \bigoplus_{n \in \nat} \opd (n) \otimes_{\sym_n} V^{\otimes n}$. 
  This yields the analytic functor $\opd (-) : \vectf \rightarrow \opdalgf$ and one has  the restriction map
 \[
 \nt _{A \in \opdalgf} (A^{\otimes n}, A)
\rightarrow 
 \nt _{V\in \vectf} (\opd (V)^{\otimes n}, \opd(V)).
 \]
between $\kring$-vector spaces of natural transformations (which, in general, is far from being an isomorphism).

The unit of the operad $\opd$ induces  $V \rightarrow \opd (V)$ and hence $V^{\otimes n} \rightarrow \opd (V)^{\otimes n}$, so that  restriction gives 
 \[
 \nt _{V\in \vectf} (\opd (V)^{\otimes n}, \opd(V))
\rightarrow 
 \nt _{V\in \vectf} (V^{\otimes n}, \opd(V)).
 \]
 
\begin{prop}
\label{prop:opd_operations}
For $\opd$ an operad and $n \in \nat$, the morphisms
\[
\opd (n) 
\stackrel{\cong}{ \rightarrow }
 \nt _{A \in \opdalgf} (A^{\otimes n}, A)
 \stackrel{\cong}{ \rightarrow } 
  \nt _{V\in \vectf} (V^{\otimes n}, \opd(V))
\]
are isomorphisms of $\sym_n\op$-modules.
\end{prop}

\begin{proof} 
Although this is well-known, the argument is outlined, since it is generalized below.

Working with analytic functors over a field of characteristic zero, 
the space $\nt _{V\in \vectf} (V^{\otimes n}, \opd(V))$ is isomorphic to $\nt _{V\in \vectf} (V^{\otimes n}, \opd(n) \otimes _{\sym_n} V^{\otimes n})$. The latter is naturally isomorphic to $\opd (n)$, by Proposition \ref{prop:schur_fully_faithful}.  The composite natural transformation is an isomorphism, by construction of the first map.

To complete the proof, it suffices to show that $ \nt _{A \in \opdalgf} (A^{\otimes n}, A)
 \rightarrow 
  \nt _{V\in \vectf} (V^{\otimes n}, \opd(V))$ is injective. This is proved by the universal example as in \cite[Section 5.2.13]{LV}. Let $V_n = \langle x_1, \ldots ,x_n \rangle$ be a $\kring$-vector space  of dimension $n$ with basis $\{x_i\}$.
 Given an $\opd$-algebra $A$ and an element $a_1 \otimes \ldots \otimes a_n$ of $A^{\otimes n}$, let $V_n \rightarrow A$ be the linear map determined by $x_i \mapsto a_i$. This induces a morphism  $\opd (V_n) \rightarrow A$ of $\opd$-algebras, since $\opd (V_n)$ is the free $\opd$-algebra on $V_n$.
 
  For a natural transformation $\theta \in   \nt _{A \in \opdalg} (A^{\otimes n}, A)$, the commutative diagram 
 \[
 \xymatrix{
 V_n^{\otimes n} 
 \ar[d]
 \ar[r]
 &
 \opd (V_n) ^{\otimes n} 
 \ar[r]^{\theta_{\opd (V_n)}} 
 &
 \opd (V_n) 
 \ar[d]
 \\
 A^{\otimes n}
 \ar[rr]_{\theta_A} 
 &
 &
 A,
 }
 \]  
  in which the vertical maps are induced by $V_n \rightarrow A$, shows that the image of $a_1 \otimes \ldots \otimes a_n$ is determined by the image of $x_1 \otimes \ldots \otimes x_n$ under $V_n^{\otimes n} \rightarrow \opd( V_n)$. The result follows. 
\end{proof}

 \subsection{Passing to PROPs}
 The above extends to considering the PROP associated to the  operad $\opd$ (or more precisely, its underlying $\kring$-linear category).  
 
 \begin{nota}
 (See \cite[Section 5.4.1]{LV}.)
 Let $\cat \opd$ denote the category associated to the  operad $\opd$. 
 \end{nota}
 
 The category $\cat \opd$ has set of objects $\nat$; for $m, n \in \nat$, $\cat \opd (m , n)$ is a $\sym_m\op \times \sym_n$-module. Moreover $\cat \opd (m, 1) = \opd (m)$ as $\sym_m \op$-modules.
 
\begin{prop} 
(Cf. \cite[Proposition 5.4.2]{LV}) 
\label{prop:funct_cat_opd}
For $A \in \ob \opdalg$, there is a  functor
$\underline{A} : \cat \opd \rightarrow \vect
$
 given on objects by $m \mapsto A^{\otimes m}$ and with morphisms acting via the $\opd$-algebra structure of $A$. This construction is natural with respect to $A$.
\end{prop} 
 
It follows that, for $m, n \in \nat$ and $A \in \ob \opdalg$, there is a  morphism of  
 $\sym_m\op \times \sym_n$-modules:
 \[
 \cat \opd (m,n) 
 \rightarrow 
  \hom_{\vect} (A^{\otimes m}, A^{\otimes n}).
 \]
 Proceeding as in the operadic case, there are restriction maps:
\[
\nt _{A \in \opdalgf} (A^{\otimes m}, A^{\otimes n})
\rightarrow
\nt _{V\in \vectf} (\opd (V)^{\otimes m}, \opd (V)^{\otimes n})
\rightarrow
\nt _{V \in \vectf} (V^{\otimes m}, \opd(V)^{\otimes n}).
\] 
 The analogue of Proposition \ref{prop:opd_operations} is the following:
 
 \begin{prop}
 \label{prop:prop_operations}
 For $\opd$ an operad and  $m , n \in \nat$, there are natural isomorphisms of  $\sym_m\op \times \sym_n$-modules:
 \[
 \cat \opd (m,n) 
\stackrel{\cong}{ \rightarrow }
  \nt _{A \in \opdalgf} (A^{\otimes m}, A^{\otimes n})
\stackrel{\cong}{ \rightarrow }
 \nt _{V \in \vectf} (V^{\otimes m}, \opd(V)^{\otimes n}).
 \]
 \end{prop}
  
\begin{proof}
The proof follows that of Proposition \ref{prop:opd_operations}, {\em mutatis mutandis}.
\end{proof}

For the application, one works in the graded setting: the $\opd$-algebra  $A$ is replaced by the suspension $s A$ of its underlying vector space. For $m \in \nat$, there is a natural isomorphism $\sym_m$-equivariant isomorphism
\[
(sA)^{\otimes m}
\cong 
s^m (  \sgn_m \otimes  A ^{\otimes m})
\]
where $\sym_m$ acts via place permutations on $(sA)^{\otimes m}$ and via the diagonal action on $\sgn_m \otimes  A ^{\otimes m}$ for the signature representation $\sgn_m$ and the place permutation action on $A^{\otimes m}$.

When considering natural transformations, one allows morphisms of non-zero degree. The above naturality statement gives the natural isomorphism of $\sym_n \times \sym_m\op$-modules:
\[
\nt _{A \in \opdalgf} ((sA)^{\otimes m} ,(sA)^{\otimes n} ) 
\cong 
s^{n-m} \big(  \sgn_n \otimes \nt _{A \in \opdalgf} (A^{\otimes m} , A^{\otimes n} ) \otimes \sgn_m \big),
\]
where the right hand side is equipped with the diagonal module structure, considering $\sgn_m$ as a $\sym_m\op$-module.

One deduces the following from Proposition \ref{prop:prop_operations}:

\begin{prop}
\label{prop:nat_sA_mn}
For $m, n \in \nat$, there are natural isomorphisms of $\sym_m\op \times \sym_n$-modules:
\[
s^{n-m}
\big( 
\sgn_n \otimes \cat \opd (m, n ) \otimes \sgn_m
\big) 
\cong 
\nt _{A\in \opdalgf} ((sA)^{\otimes m} ,(sA)^{\otimes n} ) 
\cong 
\nt _{V \in \vectf} ((sV)^{\otimes m}, (s\opd(V))^{\otimes n}).
\]
\end{prop}

\subsection{A further refinement}

Proposition \ref{prop:nat_sA_mn} can be extended by considering  polynomial functors from $\vect$ to $\vect$, as in Appendix \ref{sect:poly}. For clarity, we restrict to considering Schur functors.

\begin{nota}
For $M$ a $\sym_m\op$-module ($m \in \nat$), let $M^\dag$ denote the $\sym_m\op$-module $M \otimes \sgn_m$, equipped with the diagonal structure.
\end{nota} 

The relevance of this twisting is shown by:

\begin{lem}
\label{lem:twist}
For $M$ a $\sym_m\op$-module ($m \in \nat$), there is a natural isomorphism with respect to $W \in \ob \vect$:
\[
M \otimes_{\sym_m} (sW)^{\otimes m} 
\cong 
s^m \big( M^\dag \otimes _{\sym_m} W^{\otimes m} \big),
\]
where $\sym_m$ acts via place permutations on $(sW)^{\otimes m}$ and $W^{\otimes m}$.
\end{lem}

The following is a direct consequence of Proposition \ref{prop:nat_sA_mn}:

\begin{prop}
\label{prop:refined_prop_operations}
Let $M, N$ be finite-dimensional $\sym_m\op-$ and $\sym_n\op$-modules respectively, for $m,n \in \nat$. The isomorphism of Proposition \ref{prop:nat_sA_mn} induces natural isomorphisms:
\begin{eqnarray*}
s^{n-m} \big (\hom_\kring (M^\dag , N^\dag )\otimes_{\sym_n \times \sym_m\op} \cat \opd (m,n) \big)  
& \cong &
\nt _{A \in \opdalgf} (M \otimes_{\sym_m} (sA)^{\otimes m} , N \otimes_{\sym_n}(sA)^{\otimes n} ) 
\\
&\cong&
\nt _{V \in \vectf} (M \otimes_{\sym_m}(sV)^{\otimes m}, N \otimes_{\sym_n}(s\opd(V))^{\otimes n}),
\end{eqnarray*}
where $\hom_\kring (M^\dag , N^\dag )$ is considered as a $\sym_m \times \sym_n\op$-module with the diagonal structure.
\end{prop}

The significance of the first isomorphism is that it allows the construction of induced natural transformations for arbitrary $\opd$-algebras:

\begin{cor}
\label{cor:large_nat}
For $M, N$ as in Proposition \ref{prop:refined_prop_operations} and $B \in \ob \opdalg$, 
the natural morphism $ \cat \opd (m,n) 
 \rightarrow 
  \hom_{\vect} (B^{\otimes m}, B^{\otimes n})
$ 
induces  a $\kring$-linear map:
\[
\nt _{A \in \opdalgf} (M \otimes_{\sym_m} (sA)^{\otimes m} , N \otimes_{\sym_n}(sA)^{\otimes n} ) 
\rightarrow 
\hom_\kring (M \otimes_{\sym_m} (sB)^{\otimes m} , N \otimes_{\sym_n}(sB)^{\otimes n} )
\]
denoted $\phi \mapsto \phi_B$. This is natural in the following sense: for $B \rightarrow B'$ a morphism of $\opdalg$, the following diagram in $\vect$ commutes:
\[
\xymatrix{
M \otimes_{\sym_m} (sB)^{\otimes m}
\ar[r]^{\phi_B}
\ar[d]
&
N \otimes_{\sym_n}(sB)^{\otimes n}
\ar[d]
\\
M \otimes_{\sym_m} (sB')^{\otimes m}
\ar[r]_{\phi_{B'}}
&
N \otimes_{\sym_n}(sB')^{\otimes n}
}
\]
in which the vertical morphisms are induced by $B \rightarrow B'$.
\end{cor}

\begin{rem}
\label{rem:naturality_large_nat}
Corollary \ref{cor:large_nat} extends to treat the compatibility between the composition of natural transformations and the linear morphisms. Namely, in addition to the hypotheses of Corollary \ref{cor:large_nat}, suppose that $Q$ is a finite-dimensional $\sym_q\op$-module, so that one can consider $\nt _{A \in \opdalgf} (N \otimes_{\sym_n} (sA)^{\otimes n} , Q \otimes_{\sym_q}(sA)^{\otimes q} ) $ and $\nt _{A \in \opdalgf} (M \otimes_{\sym_n} (sA)^{\otimes n} , Q \otimes_{\sym_q}(sA)^{\otimes q} ) $, as in Proposition \ref{prop:refined_prop_operations}. Composition of natural transformations gives 
\begin{eqnarray*}
\nt _{A \in \opdalgf} (N \otimes_{\sym_n} (sA)^{\otimes n} , Q \otimes_{\sym_q}(sA)^{\otimes q} )
&\otimes &
\nt _{A \in \opdalgf} (M \otimes_{\sym_m} (sA)^{\otimes m} , N \otimes_{\sym_n}(sA)^{\otimes n} ) 
\longrightarrow \\
&&
\quad
\nt _{A \in \opdalgf} (M \otimes_{\sym_m} (sA)^{\otimes m} , Q \otimes_{\sym_q}(sA)^{\otimes q} ).
\end{eqnarray*}

Likewise, one has composition of $\kring$-linear morphisms:
\begin{eqnarray*}
\hom_\kring (N \otimes_{\sym_n} (sB)^{\otimes n} , Q \otimes_{\sym_q}(sB)^{\otimes q} )
& \otimes &
\hom_\kring (M \otimes_{\sym_m} (sB)^{\otimes m} , N \otimes_{\sym_n}(sB)^{\otimes n} )
\longrightarrow 
\\
&&
\quad 
\hom_\kring (M \otimes_{\sym_m} (sB)^{\otimes m} , Q \otimes_{\sym_q}(sB)^{\otimes q} ).
\end{eqnarray*}

Proposition \ref{prop:funct_cat_opd} together with Proposition \ref{prop:refined_prop_operations} imply that these   are compatible via the $\kring$-linear map of Corollary \ref{cor:large_nat}.
\end{rem}

\section*{Acknowledgements}
The author is very grateful to Teimuraz Pirashvili for suggesting that an earlier construction introduced by the author might be related to \cite[Conjecture 1]{2019arXiv190400121P}. Moreover, he is indebted to him for the question that lead to Theorem \ref{THM:Linf}. 

The author thanks Christine Vespa for drawing Mostovoy's results to his attention. 
 
The author is also extremely grateful to  the anonymous referee for their careful reading of the manuscript and for their useful and insightful suggestions.  
 
\section*{Funding} 

The author is employed by the French {\em Centre National de la Recherche Scientifique} (CNRS). 
 
This work was partially supported by the ANR Project {\em ChroK}, {\tt ANR-16-CE40-0003}.


\providecommand{\bysame}{\leavevmode\hbox to3em{\hrulefill}\thinspace}
\providecommand{\MR}{\relax\ifhmode\unskip\space\fi MR }
\providecommand{\MRhref}[2]{%
  \href{http://www.ams.org/mathscinet-getitem?mr=#1}{#2}
}
\providecommand{\href}[2]{#2}

\end{document}